\documentclass[12pt]{article}
\usepackage{a4wide}
\usepackage{amsmath,amssymb,amsthm}
\usepackage[mathscr]{eucal}
\usepackage{tikz}
\usetikzlibrary{positioning,automata}
\usetikzlibrary{arrows,calc}
\usetikzlibrary{decorations.markings,plotmarks}
\usepackage{graphics}
\usepackage{color}
\usepackage{hyperref}
\usepackage{dsfont}
\usepackage{enumerate}
\usepackage{mathtools}
\usepackage{cite}

\newtheorem{theorem}{Theorem}
\newtheorem{proposition}[theorem]{Proposition}
\newtheorem{corollary}[theorem]{Corollary}
\newtheorem{lemma}[theorem]{Lemma}
\newtheorem{fact}[theorem]{Fact}

\theoremstyle{definition}
\newtheorem{definition}[theorem]{Definition}

\newtheorem{conjecture}[theorem]{Conjecture}
\newtheorem{remark}[theorem]{Remark}

\newtheorem{obs}[theorem]{Observation}

\numberwithin{theorem}{section}

\DeclareMathOperator{\Aut}{Aut}

\DeclareMathOperator{\dist}{dist}
\DeclareMathOperator{\supp}{supp}
\DeclareMathOperator{\mindeg}{mindeg}
\DeclareMathOperator{\motion}{motion}

\newcommand{\footremember}[2]{%
    \footnote{#2}
    \newcounter{#1}
    \setcounter{#1}{\value{footnote}}%
}

\title{On the spectral gap and the automorphism group of distance-regular graphs}
\author{Bohdan Kivva\footremember{alley}{University of Chicago, e-mail: bkivva@uchicago.edu}}

\begin{document}
\maketitle

\vspace{-0.8cm}
\begin{abstract}
We prove that a distance-regular graph with a dominant distance is a spectral expander. The key ingredient of the proof is a new inequality on the intersection numbers. We use the spectral gap bound to study the structure of the automorphism group.
 
The minimal degree of a permutation group $G$ is the minimum number of points not fixed by non-identity elements of $G$. Lower bounds on the minimal degree have strong structural consequences on $G$. In 2014 Babai proved that the automorphism group of a strongly regular graph with $n$ vertices has minimal degree $\geq c n$, with known exceptions. Strongly regular graphs correspond to distance-regular graphs of diameter 2. Babai conjectured that  Hamming and Johnson graphs are the only primitive distance-regular graphs of diameter $d\geq 3$ whose automorphism group  has sublinear minimal degree.  We confirm this conjecture for non-geometric primitive distance-regular graphs of bounded diameter.  We also show if the primitivity assumption is removed, then only one additional family of exceptions arises, the cocktail-party graphs. We settle the geometric case in a companion paper.

\end{abstract}


\section{Introduction}

In this paper we prove a tradeoff between combinatorial and spectral parameters of distance-regular graphs and apply the result to the structure of the automorphism group.

\subsection{Main results: Spectral expansion of distance-regular graphs}

 We say that a $k$-regular graph is a \textit{spectral $\eta$-expander} for $\eta>0$, if every non-principal eigenvalue $\xi_i$ of its adjacency matrix satisfies $|\xi_i|\leq k(1-\eta)$. We say that a graph on $n$ vertices has \textit{$(1-\varepsilon)$-dominant distance} $t$, if among the $\binom{n}{2}$ pairs of distinct vertices at least $(1-\varepsilon)\binom{n}{2}$ are at distance $t$.

In our main result on spectral expansion we show that distance-regular graphs of bounded diameter are spectral expanders if they have $(1-\varepsilon)$-dominant distance for sufficiently small $\varepsilon>0$, depending only on the diameter.

\begin{theorem}\label{thm:main-spectral-intr} For every $d\geq 2$ there exist $\epsilon = \epsilon(d)>0$ and $\eta = \eta(d)>0$ such that the following holds. If a distance-regular graph $X$ of diameter $d$ has a $(1-\epsilon)$-dominant distance, then $X$ is a spectral $\eta$-expander. 
\end{theorem}

We prove this theorem in equivalent formulation as Theorem~\ref{thm:main-expansion}, and its slightly modified variant stated as Theorem~\ref{thm:main-spectral-gap} is the heart of the proof of our main result  on the minimal degree of the automorphism group (Theorem~\ref{thm:main-motion-intr}).

The key ingredient of our proof of the spectral expansion is the following (we believe new) inequality  for the intersection numbers of a distance-regular graph. This inequality provides a tradeoff between $b_{j+1}$ and $c_{j+2}$, if there are (significantly) more vertices at distance $j+1$ from a vertex $v$, than vertices at distance $j$ from $v$. In particular, this inequality implies that if $b_j$ is large and $c_{j+1}$ is small, then $c_{j+2}$ and $b_{j+1}$ cannot be small simultaneously.

\begin{theorem}[Growth-induced tradeoff]\label{b-c-ineq-intr}
Let $X$ be a distance-regular graph of diameter $d\geq 2$. Let $0 \leq j\leq d-2$. Assume $b_j>c_{j+1}$ and let $C = b_j/c_{j+1}$. Then for any $1\leq s\leq j+1$ we have
\begin{equation}
b_{j+1}\left(\sum\limits_{t=1}^{s}\frac{1}{b_{t-1}}+\sum\limits_{t=1}^{j+2-s}\frac{1}{b_{t-1}}\right)+c_{j+2}\sum\limits_{t=1}^{j+1}\frac{1}{b_{t-1}}\geq 1-\frac{4}{C-1}.
\end{equation}
\end{theorem}

Further comments on this result are in Remark~\ref{tradeoff-remark}. The proof appears in Section~\ref{sec-tradeoff} (see Theorem~\ref{b-c-ineq}).

In a distance-regular graph denote by $\lambda$ and $\mu$ the number of common neighbours of a pair of adjacent vertices and a pair of vertices at distance 2, respectively. We mention, that a result of Terwilliger \cite{ter-local}, as strengthened in \cite[Theorem 4.3.3]{BCN}, shows that any non-principal eigenvalue of a $k$-regular distance-regular graph $X$ has absolute value at most $k-\lambda$ if $\mu>1$ and $X$ is not the icosahedron. This result assures that $X$ is a spectral $\eta$-expander, if $\lambda\geq \eta k$.   We note that while both our result and Terwilliger's result provide simple sufficient combinatorial conditions for being spectral expanders, they are incomparable. In fact, our primary motivation for a spectral gap bound is an application of Lemma~\ref{mixing-lemma-tool}, where Terwilliger's gap is not sufficient. 

Additionally, we note that we do not exclude the elusive case $\mu=1$, for which almost no classification results are known, and which is known to be a difficult case in various circumstances. A remarkable example is the Bannai-Ito conjecture, where the case $\mu=1$ was the only obstacle for 30 years, and was resolved only recently in the breakthrough paper by Bang, Dubickas, Koolen and Moulton~\cite{bannai-ito}.

\subsection{Main results: Minimal degree of the automorphism group} Let $\sigma$ be a permutation of a set $\Omega$. The number of points not fixed by $\sigma$ is called the \textit{degree} of the permutation $\sigma$. Let $G$ be a permutation group on the set $\Omega$. The minimum of the degrees of non-identity elements in $G$ is called the \textit{minimal degree} of $G$. We denote this quantity by $\mindeg(G)$. The study of minimal degree goes back to Jordan \cite{Jordan} and Bochert \cite{Bochert} in 19th century. 
Lower bounds on the minimal degree have strong structural consequences on $G$. In 1934 Wielandt \cite{Wielandt} showed that a linear lower bound on the minimal degree implies a logarithmic upper bound on the \textit{thickness} of the group (the greatest $t$ for which the alternating group $A_t$ is involved as a quotient group of a subgroup of $G$ \cite{Babai-str-reg}). 
Liebeck \cite{Liebeck}, Liebeck and Saxl \cite{Liebeck-Saxl} characterized all primitive permutation groups with minimal degree less than $n/3$.

Switching from the theory of permutation groups to automorphisms of combinatorial structures, we use the following terminology \cite{RS-motion}.

\begin{definition}[Russell, Sundaram]\label{def-motion}
For a graph $X$ we use term \textit{motion} for the minimal degree of the automorphism group $\Aut(X)$, 
\begin{equation}
\motion(X) := \mindeg(\Aut(X)).
\end{equation}
\end{definition}

In 2014 Babai \cite{Babai-str-reg} proved that the motion of a strongly regular graph is linear in the number of vertices, with known exceptions.

\begin{theorem}[Babai]\label{babai-str-reg-thm}
Let $X$ be a strongly regular graph on $n\geq 29$ vertices. Then either
$$\motion(X)\geq n/8,$$
or $X$, or its complement is a Johnson graph $J(s, 2)$, a Hamming graph $H(2, s)$ or a union of cliques.
\end{theorem}

The connected strongly-regular graphs are precisely distance-regular graphs of diameter two.  For distance-regular graphs of diameter $d\geq 3$ Babai made the following conjecture.

\begin{conjecture}[Babai]\label{conj-dist-reg}

For any $d\geq 3$ there exists $\gamma_d>0$, such that for any primitive distance-regular graph $X$ of diameter $d$ with $n$ vertices either
\[ \motion(X) \geq \gamma_d n,\]
or $X$ is the Hamming graph $H(d, s)$ or the Johnson graph $J(s, d)$.
\end{conjecture}

We confirm this conjecture in a pair of papers. In the present paper we cover the case when $X$ is not geometric. Moreover, we prove an eigenvalue lower bound, which will be one of the main ingredients in the analysis of the remaining cases in the companion paper \cite{kivva-geometric}.

\begin{theorem}\label{thm:main-motion-intr}
For any $d\geq 3$ there exist $\gamma_d>0$ and a positive integer $m_d$, such that for any primitive distance-regular graph $X$ of diameter $d$ with $n$ vertices either
$$\motion(X)\geq \gamma_d n,$$
or $X$ is geometric with smallest eigenvalue at least $-m_d$.

\end{theorem}

We confirm Conjecture~\ref{conj-dist-reg} for the cases not covered by the theorem above, i.e., for geometric distance-regular graphs with bounded smallest eigenvalue,  in \cite{kivva-geometric}. 

\begin{remark}
The proof of the conjecture is split between two papers\footnote{This paper significantly overlaps with the author's earlier preprint arXiv:1802.06959. The new results obtained since that posting warrant the current reorganization of the material. In particular, that preprint did not address the imprimitive case, and none of the main results of the companion paper \cite{kivva-geometric} were available.} due to the different nature of techniques and results obtained. In particular, in \cite{kivva-geometric} we obtain a classification of geometric distance-regular graphs with bounded smallest eigenvalue under certain combinatorial and spectral assumptions, but with no assumption on the motion.
\end{remark}

In Section~\ref{sec:imprimitive} we analyze imprimitive distance-regular graphs and show that if the primitivity assumption in Conjecture~\ref{conj-dist-reg} is removed then only one new family of exceptions arises. More precisely, we show that any imprimitive distance-regular graph with sublinear motion is antipodal and its (primitive) folded graph has sublinear motion as well. Note that all antipodal covers of the Hamming graphs and Johnson graphs are known \cite{antipodal-covers} (see Theorem~\ref{thm:classical-covers} below). Thus we prove that Conjecture~\ref{conj-dist-reg} will automatically extend to the imprimitive case as follows.

\begin{theorem}\label{thm:main-imprimitive-intr} Assume Conjecture~\ref{conj-dist-reg} is true. Then for any $d\geq 3$ there exists $\widetilde{\gamma}_d>0$, such that for any distance-regular graph $X$ of diameter $d$ on $n$ vertices either
$$\motion(X)\geq \widetilde{\gamma}_d n,$$
or $X$ is a Johnson graph $J(s, d)$, or a Hamming graph $H(d, s)$, or a cocktail-party graph. 
\end{theorem}

\begin{remark} We will give a proof of Conjecture~\ref{conj-dist-reg} in \cite{kivva-geometric}, so the theorem above will turn into an unconditional result.  
\end{remark}

\subsection{Graphs with bounded smallest eigenvalue}

In Theorem~\ref{thm:main-motion-intr} we show that the exceptions with sublinear motion have bounded smallest eigenvalue. A number of classification results is known under the assumption of bounded smallest eigenvalue. 

For strongly regular graphs,  Neumaier~\cite{Neumaier}  showed in 1979 that if the smallest eigenvalue is $-m$ (for $m\geq 2$), then it is a Latin square graph $LS_m(n)$, a Steiner graph $S_m(n)$, complete multipartite graph or one of finitely many other graphs. A classification of the strongly regular graphs with smallest eigenvalue $-2$ was known earlier (Seidel~\cite{Seidel}) (1968) . Moreover, in 1976 \cite{regular-classif}  Cameron,  Goethals,  Seidel and  Shult gave a complete classification of all graphs with smallest eigenvalue $-2$. They proved that all but finitely many of such graphs have rich geometric structure (they are generalized line graphs). 

Bang and Koolen \cite{Bang-Koolen-conj} proved that all but finitely many distance-regular graphs with smallest eigenvalue $-m$ and $\mu\geq 2$ are geometric. For geometric distance-regular graphs with smallest eigenvalue $\geq -3$ and $\mu \geq 2$ Bang~\cite{Bang-diam-3} and  Bang, Koolen~\cite{Bang-Koolen-diam-3} gave a complete classification.  Moreover, they conjectured {\cite[Conjecture 7.4]{Bang-Koolen-conj}} that for any integer $m$ all but finitely many geometric distance-regular graphs with smallest eigenvalue $-m$ and $\mu\geq 2$ are known.

\begin{conjecture}[Bang, Koolen]\label{conj-bang-k}
For a fixed integer $m\geq 2$, any geometric distance-regular graph with smallest eigenvalue $-m$, diameter $\geq 3$ and $\mu\geq 2$ is either a Johnson graph, or a Hamming graph, or a Grassmann graph,  or a bilinear forms graph, or the number of vertices is bounded above by a function of $m$.
\end{conjecture}

In \cite{kivva-geometric} we confirm the above conjecture in the very  special case, when the second largest eigenvalue is sufficiently close to $b_1$, and some distance $i$ is dominant. We note that all families mentioned in Conjecture~\ref{conj-bang-k} have dominant distance $d$. Among these families, only the Hamming and the Johnson graphs satisfy the assumption that the second largest eigenvalue is close to $b_1$.

\subsection{The main tools to bound motion}\label{sec-tools} 

To obtain lower bounds on motion we follow the approach proposed by Babai in \cite{Babai-str-reg}. In that paper he used a combination of an old combinatorial tool \cite{Babai-annals} and a new spectral tool.  We combine his tools with a structural theorem of Metsch.

\subsubsection{Combinatorial tool}

\begin{definition}[Babai \cite{Babai-annals}]\label{def-disting}
A pair of vertices $u$ and $v$ is \textit{distinguished} by a vertex $x$ in a graph $X$ if the distances $\dist(x,u)$ and $\dist(x,v)$ in $X$ are distinct.
\end{definition}

\begin{definition}\label{def:mindistnum} In any graph define $D(u,v)$ to be the number of vertices that distinguish $u$ and $v$. Define the \textit{minimal distinguishing number} $D_{\min}(X)$ of the graph $X$ to be 
\[D_{\min}(X) = \min\limits_{u\neq v\in V} D(u, v).\]
\end{definition}

A simple observation shows that the minimal distinguishing number gives a lower bound on the motion of a graph.

\begin{lemma}\label{obs1}
Let $X$ be  a graph with $n$ vertices. If each pair of distinct vertices $u,v$ of $X$ is distinguished by at least $m$ vertices, then $\motion(X)\geq m$.
\end{lemma} 
\begin{proof} Indeed, let $\sigma\in \Aut(X)$ be any non-trivial automorphism of $X$. Let $u$ be a vertex not fixed by $\sigma$. No fixed point of $\sigma$ distinguishes $u$ and $\sigma(u)$, so the degree of $\sigma$ is $\geq m$.
\end{proof}

%
%

\subsubsection{Spectral tool}

For a $k$-regular graph $X$ let $k = \xi_1\geq \xi_2\geq...\geq \xi_n$ denote the eigenvalues of the adjacency matrix of $X$. Following \cite{Babai-str-reg}, we call $\xi = \xi(X) = \max\{ |\xi_i|: 2\leq i\leq n\}$ the \textit{zero-weight spectral radius} of $X$.  We refer to the quantity $(k-\xi)$ as the \textit{spectral gap} of $X$. Note that $\xi(X)\leq (1-\eta)k$ if and only if $X$ is a spectral $\eta$-expander.

The next lemma gives a lower bound on the motion of a regular graph $X$ in terms of the zero-weight spectral radius and the maximum number of common neighbors for a pair of distinct vertices in $X$.

\begin{lemma}[Babai, {\cite[Proposition~12]{Babai-str-reg}}]\label{mixing-lemma-tool}
Let $X$ be a regular graph of degree $k$ on $n$ vertices with zero-weight spectral radius $\xi$. Suppose every pair of vertices in $X$ has at most $q$ common neighbors. Then 
\[\motion(X)\geq n\cdot \frac{(k-\xi-q)}{k}.\]
\end{lemma}

Note that for bipartite graphs the zero-weight spectral radius equals $k$, the degree of a vertex. So Lemma~\ref{mixing-lemma-tool} cannot be applied. We prove a bipartite analog in Section~\ref{sec:bip-spectral-tool}.

\subsubsection{Structural tool}

Together with the two tools mentioned, an important ingredient of our proofs is Metsch's geometricity criteria (see Theorem~\ref{Metsch}). Intuitively, Metsch's criteria states that if the number $\lambda$ of common neighbors of a pair of adjacent vertices is much larger than the number $\mu$ of common neighbors of a pair of vertices at distance 2, then the graph has a  clique geometry (see Definition \ref{def:clique-geom}).

\subsubsection{How do we apply these tools?}  

We show that if a distance-regular graph $X$ is primitive and not a spectral expander, then the minimal distinguishing number is linear in the number of vertices. Thus we can apply the Combinatorial tool. 

Hence, we may assume that the graph $X$ is a spectral expander. If $\lambda$ and $\mu$ are small, a linear lower bound on motion follows from the Spectral tool. If $\mu$ is large, we again show that the minimal distinguishing number is linear in the number of vertices.

Finally, in the remaining case when $\lambda$ is large and $\mu$ is small, we use the Metsch criteria to deduce geometricity of $X$.

\subsection{Organization of the paper}

Basic concepts and definitions are given in Section~\ref{sec-prelim}. In Section~\ref{sec-drg-approx} we prove the spectral gap bound for distance-regular graphs with dominant color. More specifically, in Section~\ref{sec:spec-approx} we use the perturbation theory of matrices to show that under modest assumptions on the intersection numbers, the spectrum of a graph can be approximated by the intersection numbers $a_i$. In Section~\ref{sec-tradeoff} we prove a growth-induced tradeoff for the intersection numbers (Theorem~\ref{b-c-ineq-intr}). In Section~\ref{sec:spectral-gap} we use this inequality to derive our main result on spectral expansion (Theorem~\ref{thm:main-spectral-intr}).

In Section~\ref{sec-general} we analyze the motion of distance-regular graphs and prove our main result about automorphism groups (Theorem~\ref{thm:main-motion-intr}).

Finally, in Section~\ref{sec:imprimitive} we prove a bipartite analog of the Spectral tool and show that this is sufficient to prove Conjecture~\ref{conj-dist-reg}  to obtain a similar classification for imprimitive graphs (Theorem \ref{thm:main-imprimitive-intr}).  

\section*{Acknowledgments}

The author is grateful to Professor L\'aszl\'o Babai for introducing him to the problems discussed in the paper and suggesting possible ways of approaching them, his invaluable assistance in framing the results, and constant support and encouragement.

\section{Preliminaries}\label{sec-prelim}
In this section we introduce distance-regular graphs and other related definitions and concepts that will be used throughout the paper. For more about distance-regular graphs we refer the reader to the monograph \cite{BCN} and the survey  article \cite{Koolen-survey}. 

\subsection{Basic concepts and notation for graphs}

Let $X$ be a graph. We will always denote by $n$ the number of vertices of $X$ and if $X$ is regular we denote by $k$ its degree. Denote by $\lambda = \lambda(X)$ the minimum number of common neighbors for pairs of adjacent vertices in $X$. Denote by $\mu = \mu(X)$ the maximum number of common neighbors for pairs of vertices at distance 2. We will denote the diameter of $X$ by $d$. If the graph is disconnected, then its diameter is defined to be $\infty$. Denote by $q(X)$ the maximum number of common neighbors of two distinct vertices in $X$.

Let $N(v)$ be the set of neighbors of vertex $v$ in $X$ and $N_i(v) = \{ w\in X| \dist(v,w) = i\}$ be the set of vertices at distance $i$ from $v$ in the graph $X$.

Let $A$ be the adjacency matrix of $X$. Suppose that $X$ is $k$-regular. Then the all-ones vector is an eigenvector of $A$ with eigenvalue $k$. We will call them the \textit{trivial (principal) eigenvector} and the \textit{trivial (principal) eigenvalue}. All other eigenvalues of $A$ have absolute value not greater than $k$. We call them \textit{non-trivial (non-principal)}  eigenvalues.

\subsection{Distance-regular graphs}\label{sec:dist-reg-graphs}
\begin{definition}
 A connected graph $X$ of diameter $d$ is called \textit{distance-regular} if for any $0\leq i\leq d$ there exist constants $a_i, b_i, c_i$ such that for any $v\in X$ and any $w\in N_i(v)$ the number of edges between $w$ and $N_i(v)$ is $a_i$, between $w$ and $N_{i+1}(v)$ is $b_i$, and between $w$ and $N_{i-1}(v)$  is $c_i$. The sequence
 \[\iota(X) = \{b_0, b_1,\ldots, b_{d-1}; c_1, c_2,\ldots, c_d\}\]
 is called the \textit{intersection array} of $X$.
\end{definition}

 Note, that for a distance-regular graph $b_d = c_0 = 0$, $b_0 = k$, $c_1 = 1$, $\lambda = a_1$ and $\mu = c_2$. By edge counting, the following straightforward properties of the parameters of a distance-regular graph hold.
\begin{enumerate}
\item $a_i+b_i+c_i = k$ for every $0\leq i\leq d$,
\item $|N_i(v)|b_i = |N_{i+1}(v)|c_{i+1}$, \quad $\Rightarrow \quad $ $k_i := |N_i(v)|$ does not depend on vertex $v\in X$.
\item $b_{i+1} \leq b_{i}$ and $c_{i+1}\geq c_{i}$ for $0\leq i\leq d-1$.
\end{enumerate}

With any graph of diameter $d$ we can naturally associate matrices $A_{i}\in M_n(\mathbb{R})$, for which rows and columns are marked by vertices, with entries $(A_i)_{u,v} = 1$ if and only if $\dist(u,v) = i$. That is, $A_i$ is the adjacency matrix of the distance-$i$ graph $X_i$ of $X$. For a  distance-regular graph they satisfy the relations
\begin{equation}
A_0 = I,\quad A_1 =: A,\quad \sum\limits_{i=0}^d A_i = J,
\end{equation}
\begin{equation}\label{eq-rec}
 AA_i = c_{i+1}A_{i+1}+a_iA_i+b_{i-1}A_{i-1}\quad \text{for } 0\leq i\leq d,
 \end{equation}  
where $c_{d+1} = b_{-1} = 0$ and $A_{-1} = A_{d+1} = 0$. Clearly, Eq.~\eqref{eq-rec} implies that for every $0\leq i\leq d$ there exists a polynomial $\nu_i$ of degree exactly $i$, such that $A_i = \nu_i(A)$. Moreover, the minimal polynomial of $A$ has degree exactly $d+1$. Hence, since $A$ is symmetric, $A$ has exactly $d+1$ distinct real eigenvalues. Additionally, we conclude that for all $0\leq i,j,s\leq d$ there exist \textit{intersection numbers} $p_{i,j}^{s}$, such that 
\[A_iA_j = \sum\limits_{s = 0}^{d}p_{i,j}^{s} A_s.\]
The definition of $A_i$ implies that for any $u,v\in X$ with $\dist(u,v) = s$ there exist exactly $p_{i,j}^{s}$ vertices at distance $i$ from $u$ and distance $j$ from $v$, i.e., $|N_i(u)\cap N_j(v)| = p_{i,j}^{s}$.

Note, that 
\[p_{1,i}^{i-1} = b_{i-1},\quad p_{1,i}^{i} = a_i,\quad p_{1,i}^{i+1} = c_{i+1}.\] 

\begin{definition} A distance-regular graph $X$ of diameter $d$ is called \textit{primitive} if for each $i\in [d]$ the distance-$i$ graph $X_i$ of $X$ is connected. 
\end{definition}

Let $\eta$ be an eigenvalue of $A$, then Eq.~\eqref{eq-rec} implies that 
\[\eta \nu_i(\eta) = c_{i+1}\nu_{i+1}(\eta)+a_i\nu_i(\eta)+b_{i-1}\nu_{i-1}(\eta), \text{ so}\]
\[\eta u_i(\eta) = c_{i}u_{i-1}(\eta)+a_i u_i(\eta)+b_{i}u_{i+1}(\eta),\]
where $u_i(\eta) = \frac{\nu_i(\eta)}{k_i}$. Therefore, the eigenvalues of $A$ are precisely the eigenvalues of tridiagonal $(d+1)\times (d+1)$ \textit{intersection matrix} 

\[T(X) = \left(\begin{matrix} 
a_0 & b_0 & 0 & 0 & ... \\ 
c_1 & a_1 & b_1 & 0 &...\\
0& c_2 & a_2 & b_2 & ...\\
...& & \vdots & & ...\\
...& & 0& c_d & a_d  

\end{matrix}\right).\]

\subsection{Johnson and Hamming graphs}\label{sec-subsec-exceptions}

In this section we define families of graphs with huge automorphism groups. In particular, for certain range of parameters they have motion sublinear in the number of vertices.

\begin{definition}
Let $d\geq 2$ and $\Omega$ be a set of $m\geq 2d+1$ points. The \textit{Johnson graph} $J(m,d)$ is a graph on a set $V(J(m,d)) = \binom{\Omega}{d}$ of $n = \binom{m}{d}$ vertices, for which two vertices are adjacent if and only if the corresponding subsets $U_1, U_2\subseteq \Omega$ differ by exactly one element, i.e., $|U_1\setminus U_2| = |U_2\setminus U_1| = 1$.
\end{definition}

It is not hard to check that $J(m, d)$ is a distance-regular graph of diameter $d$ with intersection numbers
\[b_i = (d-i)(m-d-i) \quad \text{and} \quad c_{i+1} = (i+1)^{2}, \quad \text{for } 0\leq i<d.\] 
In particular, $J(m, d)$ is regular of degree $k = d(m-d)$ with $\lambda = m-2$ and $\mu = 4$.  The eigenvalues of $J(m,d)$ are
\[ \xi_j = (d-j)(m-d-j)-j \quad \text{with multiplicity}\quad \binom{m}{j} - \binom{m}{j-1}, \, \text{for } 0\leq j\leq d.\]

The automorphism group of $J(m,d)$ is the induced symmetric group $S_m^{(d)}$, which acts on $\binom{\Omega}{d}$ via the induced action of $S_m$ on $\Omega$. Indeed, it is clear, that $S^{(d)}_m\leq \Aut(J(m,d))$. The opposite inclusion can be derived from the Erd\H{o}s-Ko-Rado theorem.

Thus, for a fixed $d$ we get that the order is $|\Aut(J(m,d)| = m! = \Omega(\exp(n^{1/d}))$, the thickness satisfies $\theta(\Aut(J(m, d))) = m = \Omega(n^{1/d})$ and  
$$\motion(J(m,d)) = O(n^{1-1/d}).$$


\begin{definition}
Let $\Omega$ be a set of $m\geq 2$ points. The \textit{Hamming graph} $H(d, m)$ is a graph on a set $V(H(d, m)) = \Omega^{d}$ of $n = m^d$ vertices, for which a pair of vertices is adjacent if and only if the corresponding $d$-tuples $v_1, v_2$ differ in precisely one position. In other words, if the Hamming distance $d_H(v_1, v_2)$ for the corresponding tuples equals 1.
\end{definition}
Again, it is not hard to check that $H(d,m)$ is a distance-regular graph of diameter $d$ with intersection numbers
\[ b_i =(d-i)(m-1) \quad \text{and} \quad c_{i+1} = i+1, \quad \text{for } 0\leq i\leq d-1.\]
In particular, $H(d,m)$ is regular of degree $k = d(m-1)$ with $\lambda = m-2$ and $\mu = 2$. The eigenvalues of $H(d, m)$ are
\[ \xi_j = d(m-1) - jm \quad \text{with multiplicity}\quad \binom{d}{j}(m-1)^{j}, \, \text{for } 0\leq j\leq d.\]

The automorphism group of $H(d, m)$ is isomorphic to the wreath product $S_m\wr S_d$. Hence,  its order is $|\Aut(H(d, m))| = (m!)^d d!$, the thickness satisfies $\theta(H(d,m)) \geq m = n^{1/d}$ and $$\motion(H(d, m))\leq 2m^{d-1} = O(n^{1-1/d}).$$


\subsection{Geometric distance-regular graphs}\label{sec-geom}

In this section we discuss geometric distance-regular graphs. More on geometric distance-regular graphs can be found in~\cite{Koolen-survey} and \cite{delsarte-geom}.

Let $X$ be a distance-regular graph, and $\theta_{\min}$ be its smallest eigenvalue. Delsarte proved in \cite{Delsarte} that any clique $C$ in $X$ satisfies $|C|\leq 1-k/\theta_{\min}$. A clique in $X$ of size $1-k/\theta_{\min}$ is called a \textit{Delsarte clique}.
\begin{definition}\label{def-geom}
A distance-regular graph $X$ is called \textit{geometric} if there exists a collection $\mathcal{C}$ of Delsarte cliques such that every edge of $X$ lies in precisely one clique from $\mathcal{C}$.
\end{definition}

\begin{definition}\label{def:clique-geom} We say that a graph $X$ contains a \textit{clique geometry}, if there exists a collection $\mathcal{C}_0$ of maximal cliques, such that every edge is contained in precisely one clique from~$\mathcal{C}_0$. 
\end{definition}

Metsch proved that a graph $X$ under rather modest assumptions contains a clique geometry.

\begin{theorem}[Metsch {\cite[Result 2.2]{Metsch}}]\label{Metsch}
Let $\mu$, $\lambda^{(1)}, \lambda^{(2)}$ and $m$ be positive integers. Assume that $X$ is a connected graph with the following properties.
\begin{enumerate}
\item Every pair of adjacent vertices has at least $\lambda^{(1)}$ and at most $\lambda^{(2)}$ common neighbors.
\item Every pair of non-adjacent vertices has at most $\mu$ common neighbors.
\item \quad $2\lambda^{(1)}-\lambda^{(2)}>(2m-1)(\mu-1)-1$.
\item Every vertex has fewer than $(m+1)(\lambda^{(1)}+1)-\frac{1}{2}m(m+1)(\mu-1)$ neighbors.
\end{enumerate}

Define a \emph{line} to be a maximal clique $C$ satisfying $|C|\geq \lambda^{(1)}+2-(m-1)(\mu-1)$. Then every vertex is on at most $m$ lines, and every pair of adjacent vertices lies in a unique line. 

\end{theorem}

\begin{remark}\label{geom-smallest-eigenvalue} Suppose that $X$ satisfies the conditions of the previous theorem. Then the smallest eigenvalue of $X$ is at least $-m$.
\end{remark}
\begin{proof}
Let $\mathcal{C}$ be the collection of lines of $X$. Consider $|V|\times |\mathcal{C}|$ vertex-clique incidence matrix $N$. That is, $(N)_{v,C} = 1$ if and only if $v\in C$ for $v\in X$ and $C\in \mathcal{C}$. Since every edge belongs to exactly one line, we get $NN^{T} = A+D$, where $A$ is the adjacency matrix of $X$ and $D$ is a diagonal matrix. Moreover, $(D)_{v,v}$ equals to the number of lines that contain $v$. By the previous theorem, $D_{v,v} \leq m$ for every $v\in X$. Thus, 
\[A+mI = NN^{T}+(mI-D)\] 
is positive semidefinite, so all eigenvalues of $A$ are at least $-m$.
\end{proof}
%

We will use the following sufficient conditions of geometricity.
\begin{proposition}[{\cite[Proposition 9.8]{Koolen-survey}}]\label{suff-cond}
Let $X$ be a distance-regular graph of diameter $d$. Assume there exist a positive integer $m$ and a clique geometry $\mathcal{C}$ of $X$ such that  every vertex $u$ is contained in exactly $m$ cliques of $\mathcal{C}$. If $|\mathcal{C}|<|V(X)| = n$, then $X$ is geometric with smallest eigenvalue $-m$. Moreover, $|\mathcal{C}|<n$ holds if $\min\{|C| : C\in \mathcal{C}\}>m$.  
\end{proposition} 

\begin{corollary}[{\cite[Proposition 9.9, restated]{Koolen-survey}}]\label{geom-corol}
Let $m\geq 2$ be an integer, and let $X$ be a distance-regular graph  with $(m-1)(\lambda+1)<k\leq m(\lambda+1)$ and diameter $d\geq 2$. If $\lambda\geq \frac{1}{2}m(m+1)\mu$, then $X$ is geometric with smallest eigenvalue $-m$.
\end{corollary}
\begin{proof}
Directly follows from Theorem \ref{Metsch} and Proposition \ref{suff-cond}.
\end{proof}

\subsection{Imprimitive distance-regular graphs}

Here we shortly describe some basic properties of imprimitive distance-regular graphs that we will need later. Recall that a distance-regular graph $X$ is imprimitive if for some $1\leq i\leq d$ the distance-$i$ graph $X_i$ is disconnected. Smith's theorem states that there are only two types of imprimitive distance-regular graphs.

\begin{definition} A distance-regular graph $X$ of diameter $d$ is called \textit{antipodal} if being at distance $d$ in $X$ is an equivalence relation, namely if $X_d$ is a disjoint union of cliques.
\end{definition}

\begin{theorem}[D. H. Smith \cite{Smith}] An imprimitive distance-regular graph of degree $k>2$ is bipartite or antipodal (or both).
\end{theorem}

If $X$ is a bipartite graph, then $X_2$ has two connected components $X^{+}$ and $X^{-}$, which are called the \textit{halved graphs} of $X$ and are denoted $\frac{1}{2}X$. 

For an antipodal graph $X$ of diameter $d$, define the graph $\widetilde{X}$ which has the equivalence classes of $X_d$ as vertices and two equivalence classes are adjacent if they contain adjacent vertices. The graph $\widetilde{X}$ is called the \textit{folded graph} of $X$. 

In the next proposition we list some information about the intersection numbers of halved and folded graphs, which we will need later.

\begin{proposition}[Biggs, Gardiner \cite{biggs-gardiner}, see {\cite[Proposition 4.2.2]{BCN}}]\label{prop:impriv-intersection} Let $X$ be a distance-regular graph with intersection array $\iota(X) = \{b_0, b_1, \ldots, b_{d-1}; c_1, c_2, \ldots, c_d\}$ and diameter $d\in \{2t, 2t+1\}$. 
\begin{enumerate}
\item The graph $X$ is bipartite if and only if $b_i+c_i = k$ for $i=0, 1, \ldots, d$. In this case halved graphs are distance-regular of diameter $t$ with intersection array
\[ \iota(X^{\pm}) = \left\{\frac{b_0b_1}{\mu}, \frac{b_2b_3}{\mu}, \ldots, \frac{b_{2t-2}b_{2t-1}}{\mu}; \frac{c_1c_2}{\mu}, \frac{c_3c_4}{\mu}, \ldots, \frac{c_{2t-1}c_{2t}}{\mu}\right\}. \]
\item The graph $X$ is antipodal if and only if $b_i = c_{d-i}$ for $i\neq t$. In this case $X$ is an antipodal $r$-cover of its folded graph $\widetilde{X}$, where $r = 1+b_t/c_{d-t}$. If $d>2$, then $\widetilde{X}$ is distance-regular of diameter $t$ with intersection array
\[ \iota(\widetilde{X}) = \{b_0, b_1, \ldots b_{t-1}; c_1, c_2, \ldots, c_{t-1}, \gamma c_t\},\]
where $\gamma = r$, if $d=2t$; and $\gamma = 1$, if $d = 2t+1$. 
\end{enumerate}
\end{proposition} 

It is not hard to show that given a distance-regular graph of degree $k>2$ one may obtain a primitive distance-regular graph after halving at most once and folding at most once. More precisely, the following is true.

\begin{proposition}[see {\cite[Sec. 4.2.A]{BCN}}]\label{prop:antip-prelim} Let $X$ be a distance-regular graph of degree $k>2$.
\begin{enumerate}
\item If $X$ is a bipartite graph, then its halved graph is not bipartite.
\item If $X$ is bipartite and either has odd diameter or is not antipodal, then its halved graph is primitive.
\item If $X$ is antipodal and either has odd diameter or is not bipartite, then its folded graph is primitive.
\item If $X$ has even diameter and is both antipodal and bipartite, then the halved graphs $\frac{1}{2}X$ are antipodal, the folded graph $\widetilde{X}$ is bipartite and the graphs $\widetilde{\frac{1}{2}X}\cong \frac{1}{2}\widetilde{X}$ are primitive.
\end{enumerate}
\end{proposition}

\subsection{Approximation tool}
In Sections \ref{sec-drg-approx} we use the fact that eigenvalues of a matrix do not differ much from the eigenvalues of its small perturbation.

\begin{theorem}[Ostrowski {\cite[Appendix K]{Ostrowski}}]\label{matrix-approximation}
Let $A, B \in M_n(\mathbb{C})$. Let $\lambda_1, \lambda_2, ..., \lambda_n$ be the roots of the characteristic polynomial of $A$ and $\mu_1, \mu_2, ..., \mu_n$ be the roots of the characteristic polynomial of $B$. Consider 
\[M = \max\{|(A)_{ij}|, |(B)_{ij}|: 1\leq i, j\leq n\}, \quad \delta = \frac{1}{nM}\sum\limits_{i = 1}^n\sum\limits_{j=1}^n|(A)_{ij}-(B)_{ij}|.\]
Then, there exists a permutation $\sigma \in S_n$ such that 
\[|\lambda_i - \mu_{\sigma(i)}|\leq 2(n+1)^2M\delta^{1/n}, \quad \text{for all } 1\leq i\leq n.\]
\end{theorem}

\section{Spectral gap of a distance-regular graph}\label{sec-drg-approx}

In this section we give a bound on the spectral gap of a distance-regular graph in terms of its intersection numbers. The spectral gap bound will be used in Sections~\ref{sec-general} and~\ref{sec:imprimitive} to achieve motion lower bounds through Lemma~\ref{mixing-lemma-tool} (the Spectral tool).

\subsection{Approximation of the spectrum by the intersection numbers}\label{sec:spec-approx}

Note, that if $b_i$ and $c_i$ are simultaneously small, then by monotonicity, $b_j$ for $j\geq i$ and $c_t$ for $t\leq i$ are small. Hence, the intersection matrix $T(X)$ is a small perturbation of a block diagonal matrix $N$, where one block is upper triangular and the other block is lower triangular. So the eigenvalues of $N$ are just the diagonal entries.

\begin{lemma}\label{eigenvalues-approximation}
Let $X$ be a distance-regular graph of diameter $d$. Denote by $\theta_0>\theta_1>\ldots >\theta_d$ all distinct eigenvalues of $X$. Suppose that $b_i\leq \varepsilon k$ and $c_i\leq \varepsilon k$ for some $i\leq d$ and $\varepsilon>0$. Then 
\[ |\theta_{i}-a_i|\leq 2(d+2)^2\varepsilon^{\frac{1}{d+1}}k.\] 

In particular, if furthermore $b_{i-1}\geq \alpha k$ and $c_{i+1}\geq \alpha k$, for some $\alpha >0$ (here we define $c_{d+1} = k$), 
then the zero-weight spectral radius $\xi$ of $X$ satisfies 
\begin{equation}\label{eq-spectral-gap-dist}
\xi \leq k (1-\alpha+2(d+2)^2\varepsilon^{\frac{1}{d+1}}).
\end{equation}

\end{lemma}
\begin{proof}
Consider the matrix
\[T = T(X) = \left(\begin{matrix} 
a_0 & b_0 & 0 & 0 & ... \\ 
c_1 & a_1 & b_1 & 0 &...\\
0& c_2 & a_2 & b_2 & ...\\
...& & \vdots & & ...\\
...& & 0& c_d & a_d 
\end{matrix}\right)\] 
Let $N$ be a matrix obtained from $T$ by replacing all $b_s$ and $c_t$ with 0 for $s\geq i$ and $t\leq i$. As in Theorem \ref{matrix-approximation}, consider
\[ M = \max\{|(T)_{sj}|, |(N)_{sj}|: 1\leq s, j\leq d\} = k, \]
\[ \delta = \frac{1}{(d+1)M}\sum\limits_{s = 1}^{d+1}\sum\limits_{j=1}^{d+1}|(T)_{sj}-(N)_{sj}| \leq \frac{(d+1)\varepsilon k}{(d+1)M} = \varepsilon.\]

 Observe, that the diagonal entry $a_i$ is the only non-zero entry in the $i$-th row of $N$. Furthermore, the upper-left $i\times i$ submatrix is upper triangular and $(d-i)\times (d-i)$ lower-right submatrix is lower triangular. Then the eigenvalues of $N$ are equal to $a_j$ for $0\leq j\leq d$. Thus, the first part of the statement follows from Theorem~\ref{matrix-approximation}.

The inequalities $b_{i-1}\geq \alpha k$ and $c_{i+1}\geq \alpha k$ imply that $a_j\leq k(1-\alpha)$ for $j\neq i$, while $a_i\geq (1-2\varepsilon) k$. Hence, since $k$ is an eigenvalue of multiplicity 1 of $X$, the zero-weight spectral radius of $X$ satisfies Eq.~\eqref{eq-spectral-gap-dist}.
\end{proof}

\subsection{A growth-induced tradeoff for the intersection numbers}\label{sec-tradeoff}

\begin{obs}\label{obs:triang-ineq} Let $X$ be a graph. Denote by $\deg(v)$ the degree of a vertex $v$ in $X$, and denote by $N(u, v)$ the set of common neighbors of vertices $u$ and $v$ in $X$. Then for any vertices $u, v, w$ we have
\[ |N(u, v)|+|N(u, w)|\leq \deg(u)+|N(v, w)|. \]
\end{obs}
\begin{proof} The inequality above follows from the two obvious inclusions below
\[ N(u, v) \cup N(u, w)\subseteq N(u),\qquad\quad  N(u, v) \cap N(u, w) \subseteq N(v, w).\]
\end{proof}

Next, we prove the growth-induced tradeoff for the intersection numbers. Essentially, the theorem below claims that, if for some $j$, $b_j$ is large (and therefore, by monotonicity, so are $b_i$ for $i\leq j$) and $c_{j+1}$ is small, then $b_{j+1}$ and $c_{j+2}$ cannot be small simultaneously.

\begin{theorem}[Growth-induced tradeoff]\label{b-c-ineq}
Let $X$ be a distance-regular graph of diameter $d\geq 2$. Let $0 \leq j\leq d-2$. Assume $b_j>c_{j+1}$ and let $C = b_j/c_{j+1}$. Then for any $1\leq s\leq j+1$ we have

\begin{equation}\label{eq-tradeoff-0}
b_{j+1}\left(\sum\limits_{t=1}^{s}\frac{1}{b_{t-1}}+\sum\limits_{t=1}^{j+2-s}\frac{1}{b_{t-1}}\right)+c_{j+2}\sum\limits_{t=1}^{j+1}\frac{1}{b_{t-1}}\geq 1-\frac{4}{C-1}.
\end{equation}

\end{theorem}
\begin{remark}\label{tradeoff-remark} In applications we require the right-hand side to be bounded away from zero, i.e., $C$ to be grater than some constant $>5$.  In the case when $b_j\geq \alpha k$ for some constant $\alpha>0$, each reciprocal $1/b_t$ for $t\leq j$ is at most $1/ (\alpha k)$. Thus, if the RHS is bounded away from zero and $d$ is bounded, we get a lower bound on $b_{j+1}$ or $c_{j+2}$ that is linear in $k$. We also note that $b_j/c_{j+1} = k_{j+1}/k_{j}$, where $k_j$ is the size of the sphere of radius $j$ in $X$. So the assumption says that significant growth occurs from radius $j$ to radius $j+1$. 
\end{remark}
\begin{proof}[Proof of Theorem~\ref{b-c-ineq}]
Consider the graph $Y$ with the set of vertices $V(Y) = V(X)$, where two vertices $u,v$ are adjacent if they are at distance $\dist(u,v)\leq j+1$ in $X$. We want to find the restriction on the parameters of $X$ implied by Observation~\ref{obs:triang-ineq} applied to graph $Y$ and vertices $w,v$ at distance $j+2$ in $X$. Let $\lambda_i^Y$ denote the number of common neighbors in $Y$ for a pair of vertices $u,v$ at distance $i$ in $X$ for $i\leq j+1$. Let $\mu_{j+2}^{Y}$ denote the number of common neighbors in $Y$ for a pair of vertices $u,v$ at distance $j+2$ in $X$. The monotonicity of sequences $(b_i)$ and $(c_i)$ implies $k_{i+1}\geq C k_i$ for $i\leq j$. Thus, the degree of every vertex in $Y$ satisfies
\[k^Y = \sum\limits_{i=1}^{j+1}k_{i}\leq k_{j+1}\sum\limits_{i = 0}^{j} C^{-i}\leq k_{j+1}\frac{C}{C-1}.\]
Note, that $\displaystyle{\sum\limits_{t = 0}^{d}p^i_{s, t} = k_{s}}$. Hence, we have
\[\mu_{j+2}^Y = \sum\limits_{1\leq s,t\leq j+1} p^{j+2}_{s,t}\leq 2\sum\limits_{i = 1}^{j} k_i +p^{j+2}_{j+1, j+1}\leq \frac{2}{C-1}k_{j+1}+p^{j+2}_{j+1, j+1},\] 
\[\lambda_i^Y = \sum\limits_{1\leq s, r\leq j+1} p^{i}_{r,s}\geq \sum\limits_{1\leq s\leq j+1} p^{i}_{j+1, s}= k_{j+1}-\sum\limits_{j+2\leq s\leq d} p^{i}_{j+1, s} - p_{j+1, 0}^{i}.\]
Now we are going to get some bounds on $\displaystyle{\sum\limits_{j+2\leq s\leq d} p^{i}_{j+1, s}}$. We use the following observation. Suppose, that $x, y$ are two vertices at distance $i$. Then there are exactly $\displaystyle{\prod\limits_{t = 1}^{i} c_t}$ paths of length $i$ between $x$ and $y$. Thus, $\displaystyle{\left(\prod\limits_{t = 1}^{i} c_t\right)\sum\limits_{s = j+2}^{d} p^{j+1}_{i, s}}$ equals to the number of paths of length $i$ starting at $v$ and ending at distance at least $j+2$ from $u$ and at distance $i$ from $v$, where $\dist(u,v) = j+1$ in $X$. We count such paths by considering possible choices of edges for a path at every step. At step $t$ every such path should go from $N_{t-1}(v)$ to $N_{t}(v)$, hence there are at most $b_{t-1}$ possible choices for a path at step $t$ for $1\leq t\leq i$. Moreover, since path should end up at distance at least $j+2$ from $u$, then for some $1\leq t\leq i$ path should go from $N_{j+1}(u)$ to $N_{j+2}(u)$. Therefore, the number of paths that go from $N_{j+1}(u)$ to $N_{j+2}(u)$ at step $t$ is at most $\displaystyle{\left(\prod\limits_{s = 1}^{i}b_{s-1}\right)\frac{b_{j+1}}{b_{t-1}}}$. Hence,
\[ \sum\limits_{s = j+2}^{d} p^{i}_{j+1, s} = \frac{k_{j+1}}{k_i}\sum\limits_{s = j+2}^{d} p^{j+1}_{i, s}\leq \frac{k_{j+1}}{k_i} \sum\limits_{t = 1}^{i} \left(\prod\limits_{s = 1}^{i}b_{s-1}\right)\frac{b_{j+1}}{b_{t-1}} \left(\prod\limits_{t = 1}^{i} c_t\right)^{-1} = k_{j+1}\sum\limits_{t = 1}^{i}\frac{b_{j+1}}{b_{t-1}}. \]
Thus, in particular,
\begin{equation}\label{eq-tradeoff-1}
\lambda_i^Y\geq k_{j+1}\left(1-\sum\limits_{t=1}^{i}\frac{b_{j+1}}{b_{t-1}}\right) - p_{j+1, 0}^{i}.
\end{equation}
 Similarly, 
 \[ p_{j+1, j+1}^{j+2} \leq k_{j+1}\sum\limits_{t=1}^{j+1} \frac{c_{j+2}}{b_{t-1}}. \]
Hence, 
\begin{equation}\label{eq-tradeoff-2}
\mu_{j+2}^Y \leq k_{j+1} \left(\frac{2}{C-1}+\sum\limits_{t=1}^{j+1}\frac{c_{j+2}}{b_{t-1}}\right).
\end{equation}
By applying Observation~\ref{obs:triang-ineq} to vertices $u$,  $v$, and $w$ in  $Y$, that satisfy $\dist(v,w) = j+2$, $\dist(u,v) = s$ and $\dist(w,u) = j+2-s$  in $X$, we get 
\begin{equation}\label{eq-tradeoff-3}
 k_{j+1}+\mu_{j+2}^{Y}\geq \lambda_{s}^{Y}+\lambda_{j+2-s}^{Y}.
\end{equation}
The desired inequality~\eqref{eq-tradeoff-0} follows from Eq.~\eqref{eq-tradeoff-1},~\eqref{eq-tradeoff-2} and~\eqref{eq-tradeoff-3}, as $p_{j+1,0}^i\leq 1\leq k_{j+1}/C$.
\end{proof}

\subsection{Spectral gap bound}\label{sec:spectral-gap}

In this section we prove a bound on the spectral gap of distance-regular graphs of fixed diameter $d$ with a dominant distance. We prove Theorem~\ref{thm:main-spectral-intr} in the equivalent formulation as Theorem~\ref{thm:main-expansion}. 

Our key tool is the growth-induced tradeoff proven in the previous section, which will be applied in the following setup. Assume we know lower bounds for the intersection numbers $b_i$ of the form $b_i\geq \alpha_i k$. Our goal is to get a lower bound of the similar form either for $b_{j+1}$, or for $c_{j+2}$. We will argue, that if either $c_{j+2}\leq \varepsilon k$, or $b_{j+1}\leq \varepsilon k$, for a sufficiently small $\varepsilon>0$, then either second or first summand of the LHS in inequality~\eqref{eq-tradeoff-0} is at most a $\delta$-fraction of the LHS. Hence, the other summand is at least $(1-\delta)$-fraction of the LHS and we get a linear in $k$ lower bound on either $b_{j+1}$, or $c_{j+2}$. The two sequences we define below are the coefficients in front of $k$ in the bounds we get from inequality~\eqref{eq-tradeoff-0}.

\begin{definition} Let $0\leq\delta<1$. We say that $(\alpha_i)_{i=0}^{\infty}$ is the $FE(\delta)$ sequence, if $\alpha_0 = 1$ and for $j\geq 1$ the element $\alpha_j$ is defined by the recurrence
\begin{equation}
\alpha_{j+1} = (1-\delta)\left( \sum\limits_{t=1}^{\left\lceil \frac{j+2}{2} \right\rceil}\frac{1}{\alpha_{t-1}}+\sum\limits_{t=1}^{\left\lfloor \frac{j+2}{2}\right\rfloor}\frac{1}{\alpha_{t-1}} \right)^{-1}.
\end{equation} 

Let $\widehat{\alpha} = (\alpha_i)_{i=0}^{s}$ be a sequence. We say that $\widehat{\beta} = (\beta_i)_{i=2}^{s+2}$ is the $BE(\delta, \widehat{\alpha})$ sequence, if for $j\geq 2$ the element $\beta_j$ is defined as
\begin{equation}\label{eq:beta-def}
\beta_j = (1-\delta)\left(\sum\limits_{t = 0}^{j-2}\frac{1}{\alpha_t}\right)^{-1}.
\end{equation}
If additionally, $\widehat{\alpha}$ is the $FE(\delta)$ sequence, then we will say that $\widehat{\beta}$ is the $BE(\delta)$ sequence.
\end{definition}
\begin{remark}
FE stands for ``forward expansion'' and BE stands for ``backward expansion''.
\end{remark}

Now we specify how small we expect $\varepsilon$ to be in the argument above, so that one of the summands in the LHS of~\eqref{eq-tradeoff-0} is at most a $\delta$-fraction of the LHS. 

\begin{definition}
Let $\widehat{\alpha} = (\alpha)_{i=0}^{s}$ be a decreasing sequence of positive real numbers with $\alpha_0 = 1$. Let $0< \delta< 1$, and $\widehat{\beta} = (\beta_i)_{i=2}^{s
+2}$ be the corresponding $BE(\delta, \widehat{\alpha})$ sequence.  We say that $\varepsilon>0$ is \textit{$(\delta, j, \widehat{\alpha}, d)$-compatible} for $j\leq s\leq d-2$, if $\varepsilon$ satisfies
\begin{equation}\label{eq:epsilon-def}
  \left(\frac{\alpha_j-5\varepsilon}{\alpha_j-\varepsilon} - 2\varepsilon \sum\limits_{t=1}^{j+1}\frac{1}{\alpha_{t-1}}\right)>(1-\delta)\quad  \text{and} \quad
2(d+2)^2\varepsilon^{\frac{1}{d+1}} \leq  \beta_{j+2}\delta. 
\end{equation} 
\end{definition}

Note that if $\varepsilon$ is $(\delta, j, \widehat{\alpha}, d)$-compatible for $j\geq 1$, then it is $(\delta, (j-1) , \widehat{\alpha}, d)$-compatible as well. Note also that the second condition on $\varepsilon$ implies that $\delta>\varepsilon$ and $\beta_{j+2}>\varepsilon$, $\alpha_j>\varepsilon$. 

\begin{definition} We say that $\varepsilon>0$ is $(\delta, d)$-compatible, if it is $(\delta, d-2, \widehat{\alpha}, d)$-compatible for $FE(\delta)$ sequence $\widehat{\alpha}$. We introduce notation
\[ EPS_{\delta}(d) = \sup\{\varepsilon \mid \varepsilon \text{ is $(\delta, d)$-compatible} \}.\]  
\end{definition}

In the proposition below we provide a formal version of the discussion at the beginning of this subsection. 

\begin{proposition}\label{bound-on-b}
Let $X$ be a distance-regular graph of diameter $d\geq 2$. Fix any $0<\delta<1$. Let $0 \leq j\leq d-2$ and $\widehat{\alpha} = (\alpha_i)_{i=0}^{j}$ be a decreasing sequence of positive real numbers. Consider corresponding $BE(\delta, \widehat{\alpha})$ sequence $\widehat{\beta}$ and $(\delta, j, \widehat{\alpha}, d)$-compatible $\varepsilon>0$.  Assume that the intersection numbers of $X$ satisfy $c_{j+1}\leq \varepsilon k$ and $b_i\geq \alpha_i k$ for all $0\leq i\leq j$.
 Then one of the following is true.
\begin{enumerate}
\item $b_{j+1}\geq \varepsilon k$ and $c_{j+2}\geq \varepsilon k$. 
\item The zero-weight spectral radius $\xi$ of $X$ satisfies 
\[\xi\leq k(1-(1-\delta)\beta_{j+2}),\quad  \text{and} \quad  c_{j+2}\geq \varepsilon k.\]
\item Let $\displaystyle{\alpha_{j+1} = (1-\delta)\left( \sum\limits_{t=1}^{\left\lceil \frac{j+2}{2} \right\rceil}\frac{1}{\alpha_{t-1}}+\sum\limits_{t=1}^{\left\lfloor \frac{j+2}{2}\right\rfloor}\frac{1}{\alpha_{t-1}} \right)^{-1}}$, then $b_{j+1}\geq\alpha_{j+1} k$ and $c_{j+2}\leq \varepsilon k$.
\end{enumerate}
\end{proposition}
\begin{proof} \textbf{Case 1.} Assume that $c_{j+2}\geq \beta_{j+2} k$. 

If $b_{j+1}\geq \varepsilon k$, then statement 1 holds. 
Thus, suppose that $b_{j+1}\leq \varepsilon k$. Then we fall into the assumptions of Lemma \ref{eigenvalues-approximation} with $i = j+1$. Hence, the zero-weight spectral radius $\xi$ of $X$ satisfies $$\xi\leq k (1-\min(\alpha_j, \beta_{j+2})+2(d+2)^2\varepsilon^{\frac{1}{d+1}})\leq k(1-(1-\delta)\beta_{j+2}).$$
Note, that by definition, $\beta_{j+2}<\alpha_j$, so $\min(\alpha_{j}, \beta_{j+2}) = \beta_{j+2}$.

\textbf{Case 2.} Assume $\varepsilon k\leq c_{j+2}\leq \beta_{j+2} k$. Then, by Eq.~\eqref{eq:beta-def} and Eq.~\eqref{eq:epsilon-def},
 $$\beta_{j+2} \leq \left(\frac{\alpha_j-5\varepsilon}{\alpha_j-\varepsilon}\right)\left(\sum\limits_{t=1}^{j+1}\frac{1}{\alpha_{t-1}}\right)^{-1}-2\varepsilon.$$  
 Then, by Lemma \ref{b-c-ineq} for $C = \alpha_j/\varepsilon$, we get $b_{j+1}\geq \varepsilon k$.

\textbf{Case 3.} Finally, assume that $c_{j+2}\leq \varepsilon k$. Then,  since by Eq.~\eqref{eq:epsilon-def}, 
\[0<\alpha_{j+1} \leq \left(\frac{\alpha_j-5\varepsilon}{\alpha_j-\varepsilon} - \varepsilon \sum\limits_{t=1}^{j+1}\frac{1}{\alpha_{t-1}}\right)\left( \sum\limits_{t=1}^{\left\lceil \frac{j+2}{2} \right\rceil}\frac{1}{\alpha_{t-1}}+\sum\limits_{t=1}^{\left\lfloor \frac{j+2}{2}\right\rfloor}\frac{1}{\alpha_{t-1}} \right)^{-1},
\]
Lemma \ref{b-c-ineq} for $C \geq  \alpha_j/\varepsilon$ implies $b_{j+1}\geq \alpha_{j+1}k$.
\end{proof}

As an immediate corollary we get a lower bound on $b_i$ for $i\leq t$ if $c_t\leq \varepsilon k$ is small. In the Appendix (Sec.~\ref{sec:appendix}) we give explicit lower bounds for elements of $BE(\delta)$ and $FE(\delta)$ sequences. Another corollary states that we can bound each eigenvalue of $X$ if $c_d$ is small.

\begin{corollary}\label{cor:b-inequality} Let $X$ be a distance-regular graph of diameter $d\geq 2$. Fix any $0<\delta<1$. Let $\widehat{\alpha} = (\alpha_i)_{i=0}^{\infty}$ be the $FE(\delta)$ sequence and $\varepsilon$ be $(\delta, d)$-compatible.

Assume that $c_{t}\leq \varepsilon k$ for some $t\leq d$, then $b_{i}\geq \alpha_i k$ for any $0\leq i\leq t-1$.
\end{corollary}

\begin{corollary} Fix any $0<\delta<1$. Let $X$ be a  distance-regular graph of diameter $d\geq 2$.  Denote by $\theta_0>\theta_1>\ldots >\theta_d$ all distinct eigenvalues of $X$.  Let $\widehat{\alpha} = (\alpha_i)_{i=0}^{\infty}$ be the $FE(\delta)$ sequence and $\varepsilon$ be $(\delta, d)$-compatible. 

Assume that $c_d\leq \varepsilon k$, then $\theta_i\leq (1-(1-\delta)\alpha_{d-i})k$ for any $1\leq i\leq d$. 

\end{corollary}
\begin{proof}
Follows from Corollary~\ref{cor:b-inequality} and Lemma~\ref{eigenvalues-approximation}.
\end{proof}

The next theorem is the key ingredient of the proof of our main result on the motion of distance-regular graphs (Theorem~\ref{thm:main-motion-intr}). It says that for a primitive distance-regular graph either the minimal distinguishing number is linear, or the spectral gap is large.

\begin{theorem}\label{thm:main-spectral-gap} For every $d\geq 2$ there exist $\varepsilon = \varepsilon(d)>0$ and $\eta = \eta(d)>0$ such that for any  distance-regular graph $X$ of diameter $d$ one of the following is true.
\begin{enumerate}
\item For some $0\leq i\leq d-1$, we have $b_i\geq \varepsilon k$ and $c_{i+1}\geq \varepsilon k$.
\item The zero-weight spectral radius of $X$ satisfies 
$ \xi\leq k(1-\eta)$.
\end{enumerate}
\end{theorem}
\begin{proof} Fix $\delta \in (0,1)$. Let $\widehat{\alpha} = (\alpha_i)_{i=0}^{\infty}$ be the $FE(\delta)$ sequence and $\widehat{\beta} = (\beta_i)_{i=2}^{\infty}$ be the $BE(\delta)$ sequence and $\varepsilon$ be $(\delta, d)$-compatible. Set $\eta = (1-\delta)\min(\alpha_{d-1}, \beta_{d})$.

Define $c_{d+1} = k$. Let $i$ be the unique index such that $c_{i+1}> \varepsilon k$, while $c_{i}\leq \varepsilon k$. If $b_i\geq \varepsilon k$, then first statement is true. So assume that $b_{i}\leq \varepsilon k$. By Corollary~\ref{cor:b-inequality}, for every $j\leq i-1$ we have $b_{j}\geq \alpha_j k$. Thus, by Proposition~\ref{bound-on-b}, if $i\leq d-1$, then 
\[ \xi \leq k(1-(1-\delta)\beta_{i+1})\leq k\left(1-(1-\delta)\beta_{d}\right)\leq k(1-\eta).\] 
If $i=d$, using that $b_j\geq \alpha_j k\geq \alpha_{d-1}k$ for $j\leq d-1$, we get 
\[  \xi \leq k(1-(1-\delta)\alpha_{d-1})\leq  k(1-\eta).\]
\end{proof}

\begin{remark}\label{rem:eta-bound} In the theorem above one can set 
\[\displaystyle{\eta = \frac{1}{4}d^{-(1+\log_2 d)}} \quad \text{ and } \quad \displaystyle{\varepsilon = 200^{-(d+1)}d^{-(d+1)(\log_2(d)+3)}}.\] 
\end{remark}
\begin{proof}  The proof is based on the explicit bound on the elements of $FE(\delta)$, $BE(\delta)$ and $EPS_{\delta}$ sequences given in the Appendix (Lemmas~\ref{lem:estimation-alpha} and~\ref{lem:eps-estimation}). Note that in Theorem~\ref{thm:main-spectral-gap} $\eta$ is chosen as  $\eta = (1-\delta)\min(\alpha_{d-1}, \beta_d)$. Fix $\delta = 1/9$. Using Lemma~\ref{lem:estimation-alpha} we get 
\[\displaystyle{\alpha_{d-1}\geq  \frac{(1-\delta)^2}{2} (d-1)^{-\log_2(d-1)}}\quad  \text{and} \quad \displaystyle{\beta_{d}\geq \frac{(1-\delta)^3}{2(d-1)} (d-2)^{-\log_2(d-2)}}.\] 
Hence, we obtain $\displaystyle{\eta\geq \frac{1}{4}d^{-(1+\log_2 d)}}$. Moreover, by Lemma~\ref{lem:eps-estimation}, $\displaystyle{\varepsilon = 200^{-(d+1)}d^{-(d+1)(\log_2(d)+3)}}$ is  $(\delta, d)$-compatible.
\end{proof}

Finally, we prove our main theorem on spectral expansion.

\begin{theorem}\label{thm:main-expansion} For every $d\geq 2$ there exist $\epsilon = \epsilon(d)>0$ and $\eta = \eta(d)>0$ such that the following holds. Let $X$ be a distance-regular graph of diameter $d$. If  $k_t\geq (1-\epsilon)n$ for some $t\in [d]$, then the zero-weight spectral radius of $X$ satisfies $\xi \leq k(1-\eta)$.
\end{theorem}
\begin{proof} Let $\varepsilon = \varepsilon(d)$ and $\eta = \eta(d)$ be constants provided by Theorem~\ref{thm:main-spectral-gap}. Assume that for some $i$ we have $b_i\geq \varepsilon k$ and $c_{i+1}\geq \varepsilon k$. Let $j$ be the smallest index for which $c_{j+1}\geq \varepsilon k$, then by monotonicity $b_j\geq \varepsilon k$.

If $t\leq j-1$, then $\displaystyle{k_{t+1} = \frac{b_t}{c_{t+1}}k_t\geq \frac{b_{j-1}}{c_{j}}k_t> \frac{\varepsilon k}{\varepsilon k}k_t = k_t}$. Therefore, if $k_s$ is maximal, then $s\geq j$. Observe that $k_{j+1}\geq b_j k_j/c_{j+1}\geq \varepsilon k k_j/k = \varepsilon k_j$. Moreover, if $t\geq j$, then
\begin{equation}\label{eq:degree-growth}
 k_{t+1} = \frac{b_t}{c_{t+1}}k_t\leq \frac{k}{\varepsilon k}k_t = \frac{k_t}{\varepsilon}.
 \end{equation} 
Let $k_s$ be the maximal distance degree. Note that $j<d$ as $b_j\geq \varepsilon k$. Thus, if $s = j$, then $k_{s+1}\geq \varepsilon k_s$, else $s>j$ and $k_{s-1}\geq \varepsilon k_s$, by Eq.~\eqref{eq:degree-growth}. Define $\epsilon = \epsilon(d) = \varepsilon/(1+\varepsilon)$. Hence, 
\[k_s =  n-\sum\limits_{t\neq s} k_t< n - \varepsilon k_s \quad \Rightarrow \quad k_s< \left(1-\frac{\varepsilon}{1+\varepsilon}\right)n = (1-\epsilon)n.\]
Therefore, if $k_s\geq (1-\epsilon)n$ for some $s$, then there is no $i$ such that $b_i\geq \varepsilon k$ and $c_{i+1}\geq \varepsilon k$. Hence, by Theorem~\ref{thm:main-spectral-gap}, $\xi\leq k(1-\eta)$.   
\end{proof}

\section{Motion of primitive non-geometric distance-regular graphs}\label{sec-general}

Prior to proving our main result on motion of distance-regular graphs (Theorem~\ref{thm:main-motion-intr}) in Section~\ref{sec:main-motion}, we study the minimal distinguishing number of distance-regular graphs. In the cases, when either there is no dominant distance (Proposition~\ref{primitive-distinguish}), or when the degree of a vertex is linear in the number of vertices (Proposition~\ref{prop-k-big}), we show a lower bound on the minimal distinguishing number that is linear in the number of vertices.     

\subsection{Case of a large vertex degree}\label{sec-large-deg}

In this section we study distance-regular graphs with a large vertex degree. 

\begin{lemma}\label{max-min}
Let $X$ be a distance-regular graph of diameter $d\geq 2$.
\begin{enumerate}
\item The parameters of $X$ satisfy $k-\mu\leq 2(k-\lambda)$.
\item If $a_2\neq 0$, then they also satisfy $k-\lambda\leq 2(k-\mu)$.
\end{enumerate}
\end{lemma}
\begin{proof}

The first statement follows from Observation~\ref{obs:triang-ineq} applied to vertices  $v$ and $w$ at distance 2 in $X$ and their common neighbor $u$. 

Suppose, that $a_2\neq 0$, then for a vertex $u$ there exist two adjacent vertices $v$ and $w$ at distance 2 from $u$. So, the second statement follows from Observation~\ref{obs:triang-ineq} as well. 
\end{proof}

Any pair of distinct vertices in a distance-regular graph has $\lambda$, or $\mu$, or $0$ common neighbors, if distance between them is 1, 2, or at least 3, correspondingly. Therefore, any pair of distinct vertices in a distance-regular graph is distinguished by at least $2(k-\max(\lambda, \mu))$ vertices. Combining this with the previous lemma we get the following bound.   

\begin{lemma}
Let $X$ be a distance-regular graph of diameter $d\geq 2$. Then any pair of distinct vertices is distinguished by at least $k-\mu$ vertices.
\end{lemma}
\begin{proof} Any two vertices $u,v\in X$ are distinguished by at least $|N(u)\bigtriangleup N(v)| = 2(k-|N(u)\cap N(v)|)$ vertices. Thus, by Lemma \ref{max-min}, we get $2(k-\max(\lambda,\mu))\geq k-\mu$.
\end{proof}

\begin{corollary}\label{linear-mu-bound}
Let $X$ be a distance-regular graph of diameter $d\geq 2$. Fix some constants $\gamma, \delta>0$. If $k>n\gamma$ and $\mu\leq (1-\delta)k$, then $\motion(X)\geq \gamma\delta n$. 
\end{corollary}
\begin{proof}
The result follows from the lemma above and Lemma~\ref{obs1}. 
\end{proof}

Next, we bound $\mu$ and $\lambda$ away from $k$.

\begin{lemma}\label{min-est}
The parameters of a distance-regular graph of diameter $d\geq 2$ satisfy 
\begin{enumerate}
\item $\displaystyle{\min\left(\lambda, \mu\right)<k\cdot \left(1+\min\left(\frac{r-1}{d-1}, \left(\frac{r}{d}\right)^{\frac{1}{d-1}}\right)\right)^{-1}}$,
\item $\displaystyle{\mu <k\cdot \max\left(\frac{d-1}{r-1}, \left(\frac{d}{r}\right)^{\frac{1}{d-1}}\right)}$,
\end{enumerate}
where $\displaystyle{r = \frac{n-1}{k}}$.
\end{lemma}
\begin{proof}
Recall that the sequences $(b_i)$ and $(c_i)$ are monotone, so $b_i\leq b_1 = k-\lambda-1$ and $c_{i+1}\geq \mu$ for $1\leq i\leq d-1$.
Thus, $\displaystyle{k_{i+1}\leq k_i\frac{k-\lambda-1}{\mu}}$. Hence,  $$n = \sum\limits_{i=0}^{d}k_i\leq 1+\sum\limits_{i=0}^{d-1}k\left(\frac{k-\lambda-1}{\mu}\right)^i.$$
If $k-\lambda-1\leq \mu$, then $\displaystyle{n\leq 1+k+(d-1)k\frac{k-\lambda-1}{\mu}}$, i.e., 
$$\displaystyle{\left(\frac{r-1}{d-1}\right)\mu  +\lambda+1\leq k} \quad \Rightarrow \quad \min(\lambda, \mu)< k\left(1+\frac{r-1}{d-1}\right)^{-1} \quad \text{and} \quad \mu<\frac{d-1}{r-1}k.$$ 
Otherwise, we have $\displaystyle{n\leq 1+dk\left(\frac{k-\lambda-1}{\mu}\right)^{d-1}}$, i.e., 
$$\displaystyle{\left(\frac{r}{d}\right)^{\frac{1}{d-1}}\mu+\lambda+1\leq k}\quad \Rightarrow \quad \min(\lambda, \mu)< k\left(1+\left(\frac{r}{d}\right)^{\frac{1}{d-1}}\right)^{-1}\quad \text{and}\quad \mu<\left(\frac{d}{r}\right)^{\frac{1}{d-1}}k .$$
\end{proof}

\begin{corollary}\label{min-lambda-d3}
Let $X$ be a distance-regular graph of diameter $d\geq 3$. Then $$\min(\lambda, \mu)\leq \frac{d-1}{d}k.$$
\end{corollary}
\begin{proof} Since the diameter $d\geq 3$, we have $\displaystyle{2\leq r = \frac{n-1}{k}}$, so the result follows from the previous lemma. 
\end{proof}

\begin{lemma}\label{a2-lambda}
Let $X$ be a distance-regular graph of diameter $d\geq 3$ and $a_2 = 0$. Then $\lambda = 0$.
\end{lemma}
\begin{proof}
Fix any vertex $x\in X$. Take $v\in N_2(x)$. Denote the set of neighbors of $v$ in $N_1(x)$ by $S(v)$. Then $|S(v)| = \mu$. Observe, that common neighbors of $v$ and $w\in S(v)$ are in $N_1(x)$ or in $N_2(x)$. Moreover, the condition $a_2 = 0$ force them to be in $S(v)$ and there are exactly $\lambda$ such neighbors. At the same time, $w$ has exactly $\lambda$ neighbors in $N_1(x)$, i.e., all neighbors of $w\in S(v)$ within $N_1(x)$ lie inside $S(v)$.

Suppose that $\lambda\neq 0$, then there are two adjacent vertices $w_1, w_2\in S(v)$. Moreover, for any $w_3\in N_1(x)$ we have $\dist(w_i,w_3)\leq 2$ for $i = 1,2$. And distance could not be 2 for both, as $a_2 = 0$. Therefore, $w_3$ is adjacent to at least one of $w_1, w_2$. However, by the argument above, any such $w_3$ is inside $S(v)$, i.e., $S(v) = N_1(x)$, or in another words $\mu = k$. But it means that there is no vertex at distance $3$ from $x$, so we get a contradiction with the assumtion that the diameter is at least $3$.
\end{proof}

\begin{proposition}\label{prop-k-big}
Let $X$ be a distance-regular graph of diameter $d\geq 3$. Suppose $k>n\gamma>2$ for some $\gamma>0$. If $X$ is not a bipartite graph, then $\motion(X)$ is at least $\displaystyle{\frac{\gamma}{3} n}$.
\end{proposition}
\begin{proof}
Suppose, that $\mu\leq\frac{2}{3}k$, then result follows from Corollary \ref{linear-mu-bound}. 

If diameter $d\geq 4$, then $\mu\leq \frac{k}{2}$. Indeed, let $v,w$ be vertices at distance $4$ and let $y$ be a vertex at distance 2 from each of them. Then $y$ and $v$ have $\mu$ common neighbors and $y$ and $w$ have $\mu$ common neighbors, and they are all distinct. At the same time, they all are neighbors of $y$, so $\mu\leq k/2$.

If $d = 3$ and $a_2>0$, then by Lemma \ref{max-min} the inequality $k-\min(\lambda, \mu)\leq 2(k-\max(\lambda,\mu))$ holds. Hence, any two vertices $u,v\in X$ are distinguished by at least $$|N(u)\bigtriangleup N(v)| = 2(k-|N(u)\cap N(v)|)\geq 2(k-\max(\lambda,\mu))\geq k-\min(\lambda, \mu)$$ vertices. Moreover, Corollary \ref{min-lambda-d3} for $d=3$ gives $\min(\lambda, \mu)\leq \frac{2}{3}k$. 

Finally, assume $d=3$, $a_2 = 0$ and $\mu > 2k/3>1$. Lemma 5.4.1 in \cite{BCN}, states that for a distance-regular graph of diameter $d\geq 3$, if $\mu>1$, then either $c_3\geq 3\mu/2$, or $c_3\geq \mu+b_2 = k-a_2$. Thus, we get that $c_3\geq k$, i.e., $a_3 = 0$. Hence, by Lemma \ref{a2-lambda} the graph $X$ is bipartite.
\end{proof}

\subsection{Motion of primitive distance-regular graphs}\label{sec-dist-number-gen-drg}
Recall that a distance-regular graph $X$ is primitive if the distance-$i$ graph $X_i$ is connected for every $1\leq i\leq d$. Recall also that by Definition~\ref{def:mindistnum},

\[D_{\min}(X) = \min\limits_{u\neq v\in V} D(u, v).\]

It is easy to see, that for a distance-regular graph, the number $D(u,v)$ depends only on the distance $i$ between $u$ and $v$, so one can define $D(i) = D(u,v)$. We will need the following special case of the inequality between the distinduishing numbers shown by Babai \cite{Babai-annals}.

\begin{lemma}[Babai {\cite[Proposition 6.4]{Babai-annals}}]\label{babai-dist}
Let $X$ be a primitive distance-regular graph of diameter $d$. Then for any distances $1\leq i, j\leq d$ the inequality $D(j)\leq d\cdot D(i)$ holds.
\end{lemma}
\begin{proof}
Since the distance-$i$ graph $X_i$ is connected, statement follows from the triangle inequality  $D(u,v)\leq D(u,w)+D(w, v)$ for any  vertices $u,v,w$ of $X$.
\end{proof}

\begin{proposition}\label{primitive-distinguish} Let $X$ be a primitive distance-regular graph of diameter $d\geq 2$ on $n$ vertices. Fix some positive real number $\alpha>0$. Suppose that for some $1\leq j \leq d-1$ inequalities $b_j\geq \alpha k$ and $c_{j+1}\geq \alpha k$ hold. Then $$D_{\min}(X)\geq  \frac{\alpha}{d}n.$$
\end{proposition}
\begin{proof}
 Since the sequence $(b_i)$ is non-increasing, if $t\leq j$, then $a_t = k-b_t-c_t\leq (1-\alpha)k$. Similarly, the sequence $(c_i)$ is non-decreasing, so if $t>j$, then $a_t = k-b_t-c_t\leq (1-\alpha)k$.

Consider any pair of adjacent vertices $u,v\in X$.  If vertex $x$ does not distinguish $u$ and $v$, then  $\dist(u, x) = \dist(v, x) = t$ for some $1\leq t \leq d$. Note, that for a given $t$ there are $p^{1}_{t, t}$ such vertices $x$ and $$p^1_{t,t} = p^t_{t,1}\frac{k_t}{k} = k_t\frac{a_t}{k}\leq (1-\alpha)k_t.$$
Clearly, $\sum\limits_{i = 1}^{d} k_i = n-1$. Hence, any pair of adjacent vertices is distinguished by at least 
$$n - \sum\limits_{t=1}^d (1-\alpha) k_t\geq n-(1-\alpha)n = \alpha n$$ vertices. Finally, the result follows from Lemma \ref{babai-dist}.
\end{proof} 

\subsection{Reduction to the geometric graphs}\label{sec:main-motion}

In the theorem below we prove our main result  on the motion of distance-regular graphs.

\begin{theorem}\label{main-general-case}
For any $d\geq 3$ there exist $\gamma_d>0$ and a positive integer $m_d$, such that for any primitive distance-regular graph $X$ of diameter $d$ with $n$ vertices either
$$\motion(X)\geq \gamma_d n,$$
or $X$ is geometric with smallest eigenvalue at least $-m_d$.

Furthermore, one can set $\displaystyle{m_d = \left\lfloor 5d^{\log_2 d+1}\right\rfloor}$.
\end{theorem}
\begin{proof} By Theorem~\ref{thm:main-spectral-gap}, there exist constants $\varepsilon>0$ and $\eta>0$, which depend only on $d$, such that 
\begin{itemize}
\item either $b_i\geq \varepsilon k$ and $c_{i+1}\geq \varepsilon k$,
\item or the zero-weight spectral radius of $X$ satisfies $\xi\leq k(1-\eta)$.
\end{itemize}
In the first case, by Proposition~\ref{primitive-distinguish}, we obtain \[\motion(X)\geq  \frac{\varepsilon}{d}n.\]
Hence, assume that $\xi\leq k(1-\eta)$. For convenience, we additionally assume $\eta\leq 1/7$.

\noindent\textbf{Case 1.} Suppose that $\mu > \eta^3 k$. Then, by Lemma~\ref{min-est}, $n\leq \max\left(d, 2\left(\eta^{-3}d\right)^{d-1}\right)k+1$. Therefore, by Proposition~\ref{prop-k-big}, 
\[\motion(X)\geq \frac{1}{7}\left(\eta^3 d^{-1}\right)^{d-1}n.\] 

\noindent\textbf{Case 2.} Suppose that $\displaystyle{\lambda<\frac{9}{10}\eta k}$  and $\mu \leq \eta^3 k$. Then any pair of distinct vertices in $X$ has at most $q(X) = \max(\lambda, \mu)\leq 9\eta k/10$ common neighbors. Therefore, by Lemma~\ref{mixing-lemma-tool},
\[\motion(X)\geq \frac{\eta}{10} n.\] 

\noindent\textbf{Case 3.} Suppose that $\displaystyle{\lambda\geq \frac{9}{10}\eta k}$ and $\mu\leq \eta^3 k$. Let $m$ be the integer that satisfies 
\[ (m-1)(\lambda+1)< k\leq m(\lambda+1).\] 
The assumption on $\lambda$ implies $\displaystyle{m-1\leq \frac{10}{9}\eta^{-1}}$. We additionally assumed $\eta\leq 1/7$, so
\[ \frac{1}{2}m(m+1)\mu\leq \frac{1}{2}\left(\frac{10}{9}\eta^{-1}+1\right)\left(\frac{10}{9}\eta^{-1}+2\right)\mu\leq  \frac{9}{10}\eta^{-2}\mu\leq \frac{9}{10}\eta k\leq \lambda.\]
Thus, by Corollary~\ref{geom-corol}, the graph $X$ is a geometric distance-regular graph with smallest eigenvalue $-m$.

Finally, we note that we can take $\displaystyle{m_d = \left\lfloor\frac{10}{9}\eta^{-1}+1\right\rfloor}$ and 
\[\gamma_d = \min\left(\frac{\varepsilon}{d}, \frac{1}{7}\left(\eta^3 d^{-1}\right)^{d-1}, \frac{\eta}{10}\right). \]
Furthermore, by Remark~\ref{rem:eta-bound}, $m_d$ can be taken as $\displaystyle{m_d = \left\lfloor 5d^{\log_2 d+1}\right\rfloor}$.  

\begin{remark} A bit more careful computations show that one can in fact set $$m_d= \lceil\max\left(2(d-1)(d-2)^{\log_2(d-2)}, 2 (d-1)^{\log_2(d-1)}\right)\rceil.$$
In particular, for $d = 3$ this estimate gives upper bound $m_d\leq 4$.
\end{remark}

\end{proof}

\section{Imprimitive case}\label{sec:imprimitive}

In this section we analyze the motion of imprimitive distance-regular graphs. We start by establishing a version of the Spectral tool in the bipartite case. Later we show that the motion of the antipodal graphs is controlled by the motion of their folded graphs. After that we prove motion lower bounds for bipartite graphs and for imprimitive graphs of diameter 3 and 4. A separate analysis for an imprimitive graph of diameter 3 and 4 is needed due to the fact that its folded or halved graph may be a complete graph, and in this case different arguments are required.

\subsection{Spectral tool for bipartite graphs}\label{sec:bip-spectral-tool}

To prove an analog of the Spectral tool (Lemma~\ref{mixing-lemma-tool}) for the case of bipartite graphs
we need a version of the Expander Mixing Lemma for regular bipartite graphs. 

\begin{theorem}[Expander Mixing Lemma: bipartite version, Haemers {\cite[Theorem 5.1]{Haemers}}] Let $X$ be a biregular bipartite graph with parts $U$ and $W$ of sizes $n_U$ and  $n_W$. Denote, by $d_U$ and $d_W$ the degrees of the vertices in parts $U$ and $W$, respectively. Let $\lambda_2$ be the second largest eigenvalue of the adjacency matrix $A$  of $X$. Then for any $S\subseteq U$, $T \subseteq W$
\[\left(E(S,T)\frac{n_U}{|S|} - |T|d_{W}\right)\left(E(S,T)\frac{n_W}{|T|} - |S|d_U\right)\leq \lambda_2^{2}(n_U - |S|)(n_W-|T|),\]
which, using $d_Un_U = d_Wn_W = E(U, W)$, implies

\[\left||E(S,T)| - \frac{d_W |S||T|}{n_U}\right|\leq |\lambda_2|\sqrt{|S||T|},\]
where $E(S, T)$ is the set of edges between $S$ and $T$ in $X$.
\end{theorem} 
%
%
%

Next lemma is an analog  of Lemma \ref{mixing-lemma-tool} for bipartite graphs.

\begin{lemma}\label{bip-mixing-lemma-tool}
Let $X$ be a $k$-regular bipartite graph with parts $U$ and $W$ of size $n/2$ each. Let $\lambda_2$ be the second largest eigenvalue of $A$. Moreover, suppose that any pair of distinct vertices in $X$ have at most $q$ common neighbors. Then  $$\motion(X)\geq \frac{k-|\lambda_2|-q}{2k}n.$$ 
\end{lemma}
\begin{proof}
Take any non-trivial automorphism $\sigma$ of $X$. Consider $S_1\subseteq U$ and $S_2\subseteq W$, such that  $S_1\cup S_2 = \supp(\sigma) = \{x\in X | x^{\sigma} \neq x\}$ be the support of $\sigma$. Without lost of generality, we may assume that $|S_1|\geq |S_2|$. Denote $S = S_1$ and let $T\subset W$ be a set which satisfies $S_2\subseteq T$ and $|T| = |S|$. By the Expander Mixing Lemma we get 
\[ \frac{|E(S, T)|}{|S|}\leq |\lambda_2|+k\frac{2|S|}{n}.\]
Hence, there exists a vertex $x$ in $S$ which has at most $\displaystyle{|\lambda_2|+k\frac{2|S|}{n}}$ neighbors in $T$. Thus, $x$ has at least $\displaystyle{k - \left(|\lambda_2|+k\frac{2|S|}{n}\right)}$ neighbors in $W\setminus T$, and they all are fixed by $\sigma$. Therefore, they all are common neighbors of $x$ and $x^{\sigma} \neq x$. We get the inequality $\displaystyle{q\geq k - \left(|\lambda_2|+k\frac{2|S|}{n}\right)}$, which is equivalent to 
$\displaystyle{\left(\frac{|\lambda_2|+q}{k}\right)\frac{n}{2}\geq \frac{n}{2}-|S|}$. By the definition of $S$ and $T$ the number of fixed points of $\sigma$ is at most
\[n-|S_1|-|S_2|\leq n-|S| \leq \left(\frac{1}{2}+\frac{|\lambda_2|+q}{2k}\right)n.\]
\end{proof}

\subsection{Reduction results}

We show that the motion of an imprimitive distance-regular graph is controlled by the motion of its folded or halved graph.

\begin{proposition}\label{prop:reduction-antip} Let $X$ be an antipodal distance-regular graph of diameter $d\geq 3$ on $n$ vertices and $\widetilde{X}$ be its folded graph on $\widetilde{n}$ vertices. Suppose $\motion(\widetilde{X})\geq \alpha \widetilde{n}$. Then $\motion(X)\geq \alpha n$.
\end{proposition}

\begin{proof}
Assume that $X$ is an $r$-cover of $\widetilde{X}$ and let $\phi:X \rightarrow \widetilde{X}$ be a cover map. Let $\sigma$ be an automorphism of $X$. Note that by the definition of antipodal and folded graphs, vertices of $\widetilde{X}$ are maximal cliques (connected components) of $X_d$. Since $\sigma$ is an automorphism of $X$, it preserves the relation of being at distance $d$, so $\sigma$ respects preimages of $\phi$. Hence, it induces an automorphism $\widetilde{\sigma}$ of $\widetilde{X}$ defined as $\widetilde{\sigma}(x) = \phi(\sigma(\phi^{-1}(x)))$. 

If $\widetilde{\sigma}$ is non-identity, then by the assumptions of the lemma, the degree of $\widetilde{\sigma}$ is at least $\alpha \widetilde{n}$. Suppose that $x\in V(\widetilde{X})$ is not fixed by $\widetilde{\sigma}$, then $\phi^{-1}(x)$ is disjoint from $\sigma(\phi^{-1}(x))$. Thus, all vertices in $\phi^{-1}(x)$ are not fixed by $\sigma$. Therefore, the degree of $\sigma$ is at least $r\cdot \alpha \widetilde{n} = \alpha n$.

Assume that $\widetilde{\sigma}$ is the identity map. Suppose that $\sigma$ is a non-identity map. Let $x$ be a vertex such that $\sigma(x)\neq x$ and let $y$ be adjacent to $x$. Note that $\sigma(x)$ is at distance $d$ from $x$ as $\widetilde{\sigma}$ is the identity map. Thus $\sigma(y)\neq y$, as otherwise $y$ is adjacent to $\sigma(x)$ and we get a contradiction with the assumption $d\geq 3$. Therefore, any vertex of $X$ which is adjacent to a vertex not fixed by $\sigma$ is itself not fixed  by $\sigma$. Since $X$ is connected, we get that the degree of $\sigma$ is $n$ in this case.
\end{proof}
\begin{remark} If $X$ is antipodal of diameter $d=2$, then $X$ is a complete multipartite graph and its folded graph is a complete graph. The motion of $X$ is 2 in this case, so statement of the proposition above does not hold.
\end{remark}

\begin{proposition}\label{prop:bipartite-reduction} Let $X$ be a bipartite distance-regular graph on $n$ vertices and let $X^{+}$ and $X^{-}$ be its halved graphs. Then $\motion(X)\geq \min(\motion(X^{+}), \motion(X^{-}))$. 
\end{proposition}
\begin{proof} By the definition of the halved graphs, $X$ is bipartite with parts $V(X^{+})$ and $V(X^{-})$ of the same size $n/2$. Let $\sigma$ be a non-identity automorphism of $X$. If $\sigma$ maps $X^{+}$ to $X^{-}$, then the degree of $\sigma$ is $n$. Else $\sigma$ induces automorphisms $\sigma_{+}$ and $\sigma_{-}$ of $X^{+}$ and $X^{-}$, respectively. Since $\sigma$ is non-identity, at least one of $\sigma_{+}$ and $\sigma_{-}$ is non-identity. Without lost of generality, assume $\sigma^{+}$ is non-identity. Then its degree is at least $\motion(X^{+})$. 
\end{proof}

\subsection{Bipartite graphs of diameter at least 4}

\begin{theorem}\label{thm:bipgeq4} Let $X$ be a bipartite graph of diameter $d\geq 4$ on $n$ vertices. If a halved graph of $X$ is primitive, then
\[\motion(X)\geq \gamma_{d}' n, \quad \text{where } \gamma_{d}' = (2d)^{-2d-5}.\]
\end{theorem}
\begin{proof} Let $\iota(X) = \{b_0, b_1, \ldots, b_{d-1}; c_1, c_2, \ldots, c_d\}$ be the intersection array of $X$. Denote by $Y$ a halved graph of $X$. By Proposition \ref{prop:impriv-intersection}, the intersection array of $Y$ is
\[ \iota(Y) = \left\{\frac{b_0b_1}{\mu}, \frac{b_2b_3}{\mu}, \ldots, \frac{b_{2t-2}b_{2t-1}}{\mu}; \frac{c_1c_2}{\mu}, \frac{c_3c_4}{\mu}, \ldots, \frac{c_{2t-1}c_{2t}}{\mu}\right\}, \]
where $d\in \{2t, 2t+1\}$. Note that since $X$ is bipartite, $b_i+c_i = k$, so in particular the degree of $Y$ is equal 
\[\widehat{k} = \frac{b_0b_1}{\mu} = \frac{k(k-c_1)}{\mu} = \frac{k(k-1)}{\mu}. \]

For convenience, define $c_i = k$ for any $i>d$. Take $1\leq j\leq t$ such that $c_{2j-1}\leq \varepsilon k$ and $c_{2j+1}\geq \varepsilon k$, where $\varepsilon = \left(2d\right)^{-d-2}$. Then for any $i\leq j$
\begin{equation}\label{eq:bipartite-param-proof}
\frac{c_{2i-1}c_{2i}}{\mu} \leq \frac{c_{2j-1}c_{2j}}{\mu}\leq 2\varepsilon \widehat{k} \quad \text{and} \quad   \frac{b_{2i-2}b_{2i-1}}{\mu}\geq \frac{b_{2j-1}^2}{\mu} = \frac{(k-c_{2t-1})^2}{\mu}\geq (1-\varepsilon)^2\widehat{k}.
\end{equation}
\noindent\textbf{Case 1.} Assume that $j=1$ and $b_{2j+1}\leq \varepsilon k$. Then
\[\lambda(Y) = \widehat{k} - \frac{b_2b_3}{\mu}-1\geq \widehat{k} - 2\varepsilon\widehat{k}. \]
By Lemma~\ref{max-min} we obtain that $\mu(Y)\geq \widehat{k}-4\varepsilon\widehat{k}>\widehat{k}/2$. By Lemma~\ref{min-est}, we get $|V(Y)|<t 2^t  k$. Therefore, as halved graph is not bipartite, by Proposition~\ref{prop-k-big}, we obtain 
\[ \motion(Y)\geq \frac{1}{3t2^t}|V(Y)|.\]

\noindent\textbf{Case 2.} Assume that $j\geq 2$ and $b_{2j+1}\leq \varepsilon k$. Then, for any $i\geq j+1$
\begin{equation}\label{eq:bipartite-param-proof-2}
\frac{c_{2i-1}c_{2i}}{\mu} \geq \frac{c_{2j+1}^2}{\mu} = \frac{(k-b_{2j+1})^2}{\mu} \geq (1-\varepsilon)^2 \widehat{k}.
\end{equation}

Hence, combining Eq.~\eqref{eq:bipartite-param-proof} and Eq.~\eqref{eq:bipartite-param-proof-2}, by Lemma~\ref{eigenvalues-approximation}, the zero-weight spectral radius of $Y$ satisfies
\[ \xi(Y)\leq \left(1-(1-\varepsilon)^2+2(t+2)^2\varepsilon^{\frac{1}{t+1}}\right)\widehat{k}.\]
Since $j\geq 2$, using Eq.~\eqref{eq:bipartite-param-proof}, we can estimate   
\[\lambda(Y) \leq \widehat{k} - \frac{b_2b_3}{\mu}\leq \widehat{k} - (1-\varepsilon)^2\widehat{k}\leq 2\varepsilon \widehat{k} \quad \text{and} \quad \mu(Y) = \frac{c_3c_4}{\mu}\leq 2\varepsilon \widehat{k}.\]
Therefore, by Lemma~\ref{mixing-lemma-tool} and the choice of $\varepsilon$,
\[\motion(Y)\geq \left((1-\varepsilon)^2-2(t+2)^2\varepsilon^{\frac{1}{t+1}}-2\varepsilon\right)|V(Y)|\geq \frac{1}{3}|V(Y)|.\]

\noindent\textbf{Case 3.} Assume that $b_{2j+1}> \varepsilon k$. Then 
\[ \frac{c_{2j+1}c_{2j+2}}{\mu}\geq \frac{c_{2j+1}^2}{\mu}\geq \varepsilon^2 \widehat{k} \quad \text{and} \quad \frac{b_{2j}b_{2j+1}}{\mu} \geq \frac{b_{2j+1}^2}{\mu}\geq \varepsilon^2 \widehat{k}.\]
Since $Y$ is primitive, by Proposition~\ref{primitive-distinguish}, 
\[ \motion(Y)\geq \frac{\varepsilon^2}{t}|V(Y)|.\]

Finally, since $|V(Y)| = n/2$, the inequality $\motion(X)\geq \gamma n$ follows from Proposition~\ref{prop:bipartite-reduction}  for $\displaystyle{\gamma = \min\left(\frac{1}{6t2^t}, \frac{1}{6}, \frac{(2d)^{-2d-4}}{2t}\right)\geq (2d)^{-2d-5}}$.
\end{proof}

\subsection{Bipartite antipodal graphs of diameter 4}

\begin{fact}[\cite{BCN}, p. 425]\label{fact:bi-antipodal-4} Let $X$ be a bipartite antipodal distance-regular graph of diameter $d = 4$. Then there exist $\mu$ and $m$ such that the number of vertices is $n = 2m^2\mu$, the degree is $k = m\mu$, and the intersection array is
\[ \iota(X) = \{m\mu, m\mu-1, (m-1)\mu, 1; 1, \mu, m\mu-1, m\mu\}.\]
Moreover, the spectrum of $X$ consists of $k$ and $-k$ of multiplicity 1, $\sqrt{k}$ and $-\sqrt{k}$ of multiplicity $(m-1)k$, and $0$ of multiplicity $(2k-2)$.
\end{fact}

\begin{proposition}\label{prop:bip-antip4} Let $X$ be a bipartite antipodal distance-regular graph of diameter $d=4$ on $n$ vertices. Then 
\[ \motion(X)\geq 0.15 n.\] 
\end{proposition}
\begin{proof}
Consider a pair of distinct vertices $u, v$ of $X$. If $\dist(u, v)>2$, then they are distinguished by at least $D(u, v)\geq 2k$ vertices. Since $X$ is bipartite, if $\dist(u, v) = 1$, then $D(u, v) \geq  2k$ as well. Clearly, for $u, v$ at distance 2, we have $D(u, v)\geq 2(k-\mu)$. Thus $D_{\min}(X)\geq 2(k-\mu)$.

By Fact~\ref{fact:bi-antipodal-4}, $k = m\mu$ and $n = 2m^2\mu$ for some integer $m\geq 2$. Therefore, by Lemma~\ref{obs1},
\begin{equation}\label{eq:bip-antip-1}
\motion(X)\geq D_{\min}(X)\geq \frac{m-1}{m^2}n.
\end{equation}  

At the same time, by Fact~\ref{fact:bi-antipodal-4}, we know that the second largest eigenvalue of $X$ equals $\sqrt{k}$. Then, by Lemma~\ref{bip-mixing-lemma-tool}, 
\begin{equation}\label{eq:bip-antip-2} \motion(X)\geq \frac{k-\sqrt{k}-\mu}{2k}n = \frac{m\mu-\sqrt{m\mu}-\mu}{2m\mu}n\geq\frac{m-\sqrt{m}-1}{2m}n.
\end{equation}
Using the bound given by Eq.~\eqref{eq:bip-antip-1} for $m\leq 4$, and the bound given by Eq.~\eqref{eq:bip-antip-2} for $m>4$, we get the desired inequality.
\end{proof}

\subsection{Bipartite graphs of diameter 3}\label{sec-bipartite}

\begin{fact}[\cite{BCN}, p. 432]\label{fact1}
Let $X$ be a bipartite distance-regular graph of diameter $3$. Then the number of vertices of $X$ is $n = 2+2k(k-1)/\mu$ and $X$ has the intersection array
$$\iota(X) = \{k, k-1, k-\mu; 1, \mu, k\}.$$
The eigenvalues of $X$ are $k$, $-k$ with multiplicity 1, and $\pm \sqrt{k-\mu}$ with multiplicity $\frac{n}{2}-1$. 
\end{fact}

\begin{definition}
A graph $X$ is called a \textit{cocktail-party graph} if it is obtained from a regular complete bipartite graph by deleting one perfect matching. 
\end{definition}

\begin{proposition}\label{prop:bip3}
Let $X$ be a bipartite distance-regular graph of diameter $3$. If $X$ is not a cocktail-party graph, then
\[\motion(X)\geq \frac{1}{12}n.\] 
\end{proposition}
\begin{proof}
Denote the parts of the bipartite graph $X$ by $U$ and $W$. Let $Y = X_3$ be a distance-3 graph of $X$. We consider 2 cases.

\textbf{Case 1.} Suppose that $Y$ is disconnected. Then there exists a pair of vertices $u,v$ in one of the parts, so that $u$ and $v$ lie in different connected components of $Y$. Clearly, $\dist(u,v) = 2$, so $p^{2}_{3,3} = 0$. Hence, $k_3 = 1$, and the pairs of vertices at distance $3$ form a perfect matching. Therefore, $X$ is a regular complete bipartite graph with one perfect matching deleted. 

\textbf{Case 2.} $Y$ is connected and so is itself a distance-regular graph of diameter $3$. Note, that $k+k_3 = n/2$, so if necessary, by passing to $Y$, we may assume that the degree of $X$ satisfies $k\leq n/4$. Lemma~\ref{min-est} implies $\mu<\frac{\sqrt{3}}{2}k$.  The graph $X$ is bipartite, so $\lambda=0$ and by Lemma~\ref{obs1}, 
\begin{equation}\label{eq:bip3eq1}
 \motion(X) \geq D_{\min}(X)\geq 2(k-\mu)\geq (2-\sqrt{3})k>\frac{1}{4}k.
\end{equation}
If $\mu\geq 2k/3$, then by Fact~\ref{fact1}, $n \leq 3k$, so $\motion(X)\geq n/12$. If $k/4\leq \mu<2k/3$, then $n\leq 8k$, and Eq.~\eqref{eq:bip3eq1} implies
\[ \motion(X)\geq 2(k-\mu)\geq \frac{2}{3}k\geq \frac{n}{12}.\]
Finally, assume $\mu\leq k/4$. By Fact \ref{fact1}, the second largest eigenvalue is $\lambda_2 = \sqrt{k-\mu}$. Using that $k\geq 4\mu\geq 4$ and the function $x-\sqrt{x}$ is increasing for $x\geq 1$,  
by Lemma~\ref{bip-mixing-lemma-tool},
\[\motion(X)\geq \frac{k-\mu-\sqrt{k-\mu}}{2k}n\geq \frac{3k-2\sqrt{3k}}{8k}n\geq \frac{3-\sqrt{3}}{8}n\geq \frac{n}{7}.\]
\end{proof}

\subsection{Antipodal graphs of diameter 3} 

\begin{fact}[see {\cite[p. 431]{BCN}}]\label{fact:antipodal-3} Let $X$ be an antipodal distance-regular graph of diameter $d=3$ on $n$ vertices. There exist integers $m\geq 2$, $r\geq 2$ and $t\geq 1$  such that  the following holds.
\begin{itemize}
\item If $\lambda \neq \mu$, then $n = r(k+1)$, $k = mt$, $\mu = (m-1)(t+1)/r$, $\lambda = \mu+t-m$. Moreover, the distinct eigenvalues of $X$  are $k$, $t$, $-1$ and $-m$, with multiplicities $1$, $m(r-1)(k+1)/(m+t)$, $k$, $t(r-1)(k+1)/(m+t)$, respectively.
\item If $\lambda = \mu$, then $n = r(k+1)$, $k =r\mu+1$. The distinct eigenvalues of $X$ are $k$, $\sqrt{k}$, $-1$ and $-\sqrt{k}$. 
\end{itemize} 
\end{fact}

\begin{proposition}\label{prop:antip-3} Let $X$ be an antipodal distance-regular graph of diameter $d=3$ on $n$ vertices. If $X$ is not a cocktail-party graph, then 
\[ \motion(X)\geq \frac{1}{13}n.\]
\end{proposition}
\begin{proof} 
\textbf{Case 1.} Suppose that $\lambda\neq \mu$ and $t>m$. Then $\lambda >\mu$ and so
\[ k -q(X)-\xi(X) = tm-\frac{(m-1)(t+1)}{r}-t+m-t \geq t\left(m-2-\frac{m-1}{r}\right).\]
If $m\geq 3$ and $r\geq 3$, then 
\[ k -q(X)-\xi(X) \geq t\left(m-2-\frac{m-1}{3}\right)= t\left(\frac{2}{3}m-\frac{5}{3}\right)\geq \frac{1}{9}tm = \frac{k}{9}.\]
If $r=2$ and $m\geq 4$, then 
\[ k-q(X)-\xi(X)\geq t\left(m-2-\frac{m-1}{2}\right)= t\left(\frac{1}{2}m-\frac{3}{2}\right)\geq \frac{1}{8}tm= \frac{k}{8}.\]
Therefore, in both of these situations, by Lemma~\ref{mixing-lemma-tool}, 
\[ \motion(X)\geq \frac{n}{9}.\]
If $r=2$ and $m=3$, then $n=6t+2$, $k=3t$, $\mu = t+1$ and $\lambda = 2t-2$. Note that by Lemma~\ref{obs1}, in this case
\[ \motion(X)\geq D_{\min}(X)\geq \min(2(k-\lambda), 2(k-\mu)) = 2(t+2)\geq \frac{n}{3}.\]
Finally, if $m=2$, then $k = 2t$, $\mu = (t+1)/r$ is integer, and the multiplicity of $t$ as an eigenvalue is an integer number equal $2(r-1)(2t+1)/(t+2)$. Thus we may conclude that 
\[ (t+2)\mid 2(r-1)(2t+4-3) \quad \Rightarrow \quad (t+2)\mid 6(r-1) \quad \Rightarrow \]
\[ \Rightarrow \quad (t+2)\mid 6\left(\frac{t+1}{\mu}-1\right) \quad \Rightarrow \quad (t+2)\mid 6(t+1-\mu) \quad \Rightarrow \quad (t+2)\mid 6(\mu+1). \]
Hence, in particular, 
\[(t+2)\leq 6\left(\frac{t+1}{r}+1\right), \quad \text{so}\quad (t-4)r\leq 6t+6.\]
If $t\geq 10$, we get $r\leq 11$. If $t<10$, then $r= (t+1)/\mu \leq t+1<11$.
Therefore, by Lemma~\ref{obs1}, 
\[\motion(X)\geq D_{\min}(X)\geq 2(k-\lambda) = 2\left(t+2-\frac{t+1}{r}\right)\geq \frac{r-1}{r^2}n\geq \frac{n}{13}.\]
\noindent\textbf{Case 2.} Suppose that $\lambda\neq \mu$ and $m>t$. Then $\lambda <\mu$ and so
\[ k -q(X)-\xi(X) = tm-\frac{(m-1)(t+1)}{r}-m \geq m\left(t-1-\frac{t+1}{r}\right). \]
If $r\geq 4$ and $t\geq 2$, we get
\[ k -q(X)-\xi(X) \geq m\left(t-1-\frac{t+1}{4}\right)= m\left(\frac{3}{4}t-\frac{5}{4}\right)\geq \frac{1}{8}mt = \frac{k}{8}. \]
Therefore,  by Lemma~\ref{mixing-lemma-tool}, 
\[ \motion(X)\geq \frac{n}{8}.\]
If $r\leq 4$ and $t\geq 2$, then $n \leq 4(k+1)$, and by Lemma~\ref{obs1},
\[ \motion(X)\geq D_{\min}(X)\geq \min(2(k-\lambda), 2(k-\mu)) \geq\]
 \[ \geq 2\left(mt-\frac{(m-1)(t+1)}{2}\right) \geq 2\left(mt-\frac{m(t+1)}{2}\right) \geq 2\left(mt-\frac{3mt}{4}\right) = \frac{k}{2}\geq \frac{n}{12}.\]
 Finally, if $t=1$, then $\lambda\geq 0$ implies $r=2$. Hence, we obtain  $n = 2(k+1)$, $\mu = k-1$ and $\lambda = 0$. It follows that $X$ is a cocktail-party graph in this case.
 
 \noindent\textbf{Case 3.} Assume $\lambda = \mu$. Then by Fact~\ref{fact:antipodal-3}, $n = r(k+1)$, $k = r\mu+1$ and $\xi(X) = \sqrt{k}$. By Lemma~\ref{mixing-lemma-tool}, for $r\geq 4$ 
 \[ \motion(X)\geq \frac{k-\mu-\sqrt{k}}{k}n\geq \frac{(r-1)\mu+1-\mu\sqrt{r+1}}{r\mu+1}n\geq \frac{r-\sqrt{r+1}-1}{r}n\geq \frac{n}{6}.\]
 At the same time, since $\lambda = \mu$, for $2\leq r\leq 3$,  by Lemma~\ref{obs1},
 \[ \motion(X)\geq D_{\min}(X)\geq  2(k-\mu) \geq  \frac{r-1}{r}k\geq \frac{2(r-1)}{3r^2}n\geq \frac{4}{27}n.\]
\end{proof}

\subsection{Collecting together the analysis for imprimitive graphs }

In this section we collect analysis of the imprimitive case into Theorem~\ref{thm:main-imprimitive}. In the proof we use the following result about antipodal covers proved by Van Bon and Brouwer~\cite{antipodal-covers}.

\begin{theorem}[Van Bon, Brouwer]\label{thm:classical-covers} \

\begin{enumerate}
\item The Hamming graph $H(d, s)$ has no distance-regular antipodal covers, except for $H(2, 2)$, the quadrangle, which is covered by the octagon.
\item The Johnson graph $J(s, d)$ has no distance-regular antipodal covers for $d\geq 2$.
\item The complement $\overline{J(s, 2)}$ has no distance-regular antipodal covers for $s\geq 8$.
\item The complement $\overline{H(2, s)}$ has no distance-regular antipodal covers for $s\geq 4$. 
\end{enumerate}
\end{theorem}

\begin{theorem}\label{thm:main-imprimitive} Assume Conjecture~\ref{conj-dist-reg} is true. Then for any $d\geq 3$ there exists $\widetilde{\gamma}_d>0$, such that for any distance-regular graph $X$ of diameter $d$ on $n$ vertices either
$$\motion(X)\geq \widetilde{\gamma}_d n,$$
or $X$ is a Johnson graph $J(s, d)$, or a Hamming graph $H(d, s)$, or a cocktail-party graph. 
\end{theorem}
\begin{proof} Assume that Conjecture~\ref{conj-dist-reg} is true and $\gamma_d>0$ is a constant provided by the conjecture. If $X$ is primitive, then there is nothing to prove. If $X$ is bipartite and not antipodal of diameter $d\geq 4$, then by Theorem~\ref{thm:bipgeq4} $\motion(X)\geq \gamma_d' n$. If $X$ is bipartite (possibly antipodal) graph of diameter $d=3$, then by Theorem~\ref{prop:bip3}, $X$ is either a cocktail-party graph, or $\motion(X)\geq n/12$. 

If $X$ is bipartite and antipodal of even diameter $d\geq 6$, then by Proposition~\ref{prop:antip-prelim}, folded graph $\widetilde{X}$ is bipartite (and not antipodal) of diameter $d/2$. So $$\motion(\widetilde{X})\geq \min\left(\gamma_{d/2}', \frac{1}{12}\right) |V(\widetilde{X})|$$ (we use that cocktail-party graph is antipodal). Therefore, by Proposition~\ref{prop:reduction-antip}, $$\motion(X)\geq \min\left(\gamma_{d/2}', \frac{1}{12}\right)n.$$
In the case when $X$ is bipartite and antipodal of diameter $d=4$, by Proposition~\ref{prop:bip-antip4}, $$\motion(X)\geq 0.15 n.$$

We still need to analyze the cases when $X$ is antipodal, but not bipartite, or when $X$ is antipodal of odd diameter. By Proposition~\ref{prop:antip-prelim}, in these cases, folded graph of $X$ is primitive. If diameter of $X$ is 3, then by Proposition~\ref{prop:antip-3}, $\motion(X)\geq n/13$. If diameter of $X$ is $d\geq 4$, then by Proposition~\ref{prop:antip-prelim}, folded graph $\widetilde{X}$ is primitive of diameter $\displaystyle{\widetilde{d} = \left\lfloor d/2\right\rfloor\geq 2}$. Since $\widetilde{X}$ is primitive, $\widetilde{X}$ is not the complement to a disjoint union of cliques. If $\widetilde{X}$ has at least $29$ vertices, and $X$ is not a Johnson graph $J(s, d)$, the Hamming graph $H(d, s)$, or the complement of $J(s, 2)$ or $H(2, s)$, then by Theorem~\ref{babai-str-reg-thm} and the assumption that Conjecture~\ref{conj-dist-reg} is true, $$\motion(\widetilde{X})\geq \min\left(\gamma_{\widetilde{d}}, \frac{1}{8}\right)|V(\widetilde{X})|.$$
 Therefore, in this case, by Theorem~\ref{prop:reduction-antip}, $$\motion(X)\geq \min\left(\gamma_{\widetilde{d}}, \frac{1}{8}\right) n.$$

Finally we note, that by Theorem~\ref{thm:classical-covers}, if $Y$ is the Johnson graph $J(s, d)$, the Hamming graph $H(d, s)$, or the complement of $J(s, 2)$ or $H(2, s)$ and $Y$ has at least $29$ vertices, then $Y$ has no antipodal covers. In the case when $Y$ has at most $28$ vertex, $\motion(Y)\geq |Y|/14$. So by Proposition~\ref{prop:reduction-antip}, if the graph $X$ on $n$ vertices is a distance-regular antipodal cover of such $Y$, then $\motion(X)\geq n/14$.

Taking $\displaystyle{\widetilde{\gamma} = \min\left(\gamma_d, \gamma_d', \gamma_{d/2}', \gamma_{\lfloor d/2\rfloor}, \frac{1}{14}\right)}$ we get the desired statement.  
\end{proof}

\section{Appendix: Explicit bounds for $FE(\delta)$ and $BE(\delta)$}\label{sec:appendix}

In this section we compute expilicit lower bounds for $BE(\delta)$, $FE(\delta)$ and $EPS_{\delta}$ sequences.

\begin{lemma}\label{lem:estimation-alpha} Fix $0<\delta\leq \frac{1}{9}$. Let $(\alpha_i)_{i=0}^{\infty}$ be the $FE(\delta)$ sequence and the corresponding $BE(\delta)$ sequence $(\beta_i)_{i=2}^{\infty}$. Then for $j\geq 1$
\[ \alpha_j\geq \frac{(1-\delta)^2}{2} j^{-\log_2(j)} \quad \text{and} \quad \beta_{j+2} \geq \frac{(1-\delta)^3}{2(j+1)} j^{-\log_2(j)}.\]
\end{lemma}
\begin{proof}
We prove the statement of the lemma by induction. Indeed, for $j=1,2$ we have $\alpha_1 = \frac{1-\delta}{2}$ and $\alpha_2 \geq \frac{(1-\delta)^2}{4}$, so the inequality is true. For $j\geq 2$, we have
\[\alpha_{j+1} = (1-\delta)\left( \sum\limits_{t=1}^{\lceil \frac{j+2}{2} \rceil}\frac{1}{\alpha_{t-1}}+\sum\limits_{t=1}^{\lfloor \frac{j+2}{2}\rfloor}\frac{1}{\alpha_{t-1}} \right)^{-1}\geq \frac{(1-\delta)}{j+2}\alpha_{\lceil \frac{j}{2} \rceil}\geq \]
\[ \geq  \frac{(1-\delta)^3}{2(j+2)}  \left(\frac{j+1}{2}\right)^{-\log_2\left(\frac{j+1}{2}\right)} =  \frac{(1-\delta)^3 2^{\log_2\left(\frac{j+1}{2}\right)}}{2(j+2)}  (j+1)^{-\log_2(j+1)+1}= \]
\[ = \frac{(1-\delta)(j+1)^2}{2(j+2)}\frac{(1-\delta)^2}{2} (j+1)^{-\log_2(j+1)} \geq \frac{(1-\delta)^2}{2} (j+1)^{-\log_2(j+1)}. \]
 Thus, 
 \[\beta_{j+2} = (1-\delta)\left(\sum\limits_{t = 0}^{j}\frac{1}{\alpha_t}\right)^{-1} \geq \frac{1-\delta}{j+1}\alpha_j\geq \frac{(1-\delta)^3}{2(j+1)} j^{-\log_2(j)}.\]

\end{proof}

\begin{lemma}\label{lem:eps-estimation} Let $0<\delta\leq 1/9$ and $d\geq 3$. Then $$EPS_{\delta}(d) \geq \left(\frac{\delta}{22}\right)^{(d+1)}d^{-(d+1)(3+\log_2 d)}.$$
\end{lemma}
\begin{proof} Note that for the inequality $\displaystyle{2(d+2)^2\varepsilon^{\frac{1}{d+1}} \leq \beta_{d+2}\delta}$ to be satisfied it is enough to have
\[ \varepsilon\leq \left(\frac{2^7\delta d^{-\log_2 d}}{9^3(d+1)(d+2)^2}\right)^{d+1} \leq \left(\frac{\delta\beta_{d+2}}{2(d+2)^2}\right)^{d+1}.\]
In particular, this is true if
\[ 0<\varepsilon\leq \left(\frac{\delta}{22}\right)^{(d+1)}d^{-(d+1)(3+\log_2 d)}.\]

To check that the other condition on $\varepsilon$ is satisfied, note that such choice of $\varepsilon$ satisfies $\varepsilon<\alpha_{d-2}/2$. Thus
we have
\[  \left(\frac{\alpha_{d-2}-5\varepsilon}{\alpha_{d-2}-\varepsilon} - 2\varepsilon \sum\limits_{t=1}^{d-1}\frac{1}{\alpha_{t-1}}\right)\geq 1 - 10\alpha_{d-2}^{-1}\varepsilon-2d\alpha_{d-2}^{-1}\varepsilon \geq\]
\[ \geq 1 - \frac{2\varepsilon}{(1-\delta)^2}(2d+10)d^{\log_2 d}  \geq 1 - 22d^{1+\log_2 d}\varepsilon >(1-\delta).\]
\end{proof}

\end{document}